\documentclass[a4paper,reqno]{amsart}

\usepackage[centertags]{amsmath}

\usepackage{amsmath}
\usepackage{amsfonts}
\usepackage{amssymb}
\usepackage{verbatim}
\usepackage{enumerate}

\usepackage{graphicx}
\usepackage{xcolor}

\newcommand{\calP}{{\mathcal P}}
\newcommand{\calS}{{\mathcal S}}

\newcommand{\calM}{{\mathcal M}}

\newcommand{\calT}{{\mathcal T}}

\newcommand{\calB}{{\mathcal B}}
\newcommand{\calC}{{\mathcal C}}

\newcommand{\calU}{{\mathcal U}}

\newcommand{\mP}{{\mathbb P}}

\newcommand{\F}{\mathcal{F}}
\newcommand{\G}{\mathcal{G}}
\newcommand{\I}{\mathcal{I}}

\newcommand{\N}{\mathbb N}
\newcommand{\Q}{\mathbb Q}
\newcommand{\R}{\mathbb R}

\newcommand{\Z}{\mathbb Z}

\newtheorem{theorem}{Theorem}
\newtheorem{lemma}{Lemma}
\newtheorem{remark}{Remark}

\newtheorem{example}{Example}
\newtheorem{definition}{Definition}

\title[Ising model on some planar quasi-transitive graphs]{Zero-temperature stochastic Ising model on  planar quasi-transitive graphs}

\author{Emilio De Santis}
\address{University of Rome La Sapienza, Department of Mathematics
	Piazzale Aldo Moro, 5, 00185, Rome, Italy}
\email{desantis@mat.uniroma1.it}

\author{Leonardo Lelli}
\address{University of Rome La Sapienza, Department of Mathematics
	Piazzale Aldo Moro, 5, 00185, Rome, Italy}
\email{leonardo.lelli@uniroma1.it}

\begin{document}
	\begin{abstract}
		We study the zero-temperature stochastic Ising model on some connected planar quasi-transitive  graphs, which are invariant under rotations and translations. The initial spin configuration is distributed according to a Bernoulli product measure with parameter $ p\in(0,1) $. In particular, we prove that if $ p=1/2 $ and the graph underlying the model satisfies the \emph{planar shrink property} 
then all vertices flip infinitely often almost surely.		

\medskip


	\end{abstract}

    \maketitle 
    \section{Introduction}


       In this paper, we  deal with the zero-temperature stochastic Ising model $ (\sigma_t)_{t\geq 0} $ on some connected planar quasi-transitive  graphs with homogeneous ferromagnetic interactions (see e.g. \cite{GNS2000, NNS2000}),    i.e. all the interactions are equal to a positive constant.
 The initial spin configuration is distributed according to a Bernoulli product measure with parameter $ p\in(0,1) $, see e.g. \cite{FSS2002,M2011, NNS2000}.
       The dynamic evolves in the following way: each vertex, at rate $ 1 $, changes  its spin value if it disagrees  with the majority of its neighbours and determines its spin value by a fair coin toss in case of a tie between the spins of its neighbours. This process is often referred to as \emph{domain coarsening} or \emph{majority dynamics} and it is sometimes used as an opinion model.

       A question of particular relevance is whether for each vertex $ v $ 
       its spin flips only finitely many times almost surely, i.e. in other words whether $ \sigma_t $ has an almost sure limit.    
       We say that a vertex $ v $ \emph{fixates} if the spin at $ v $ flips only finitely many times.  
       According to the classification given  in \cite{GNS2000}, 
        a model is of type $ \I $ if no site fixates almost surely, i.e all sites flip infinitely often a.s.; a model is of type $ \F $ if all sites fixate almost surely, i.e. all sites flip only finitely many times a.s. and it is said of type $ \calM $  if there are both vertices that fixate and vertices that do not fixate almost surely.
        
       The literature in the early years focused on the cubic lattice $ \Z^d $ and mainly with $ d=2 $. 
        It is known that the zero-temperature stochastic Ising model on $ \Z $ with homogeneous ferromagnetic interactions is of type $ \I $ for any initial density $ p\in (0,1) $  (see \cite{A1983,NNS2000}).
       
       The disordered model on $ \Z^d $,  if the 
interactions $ \{ J_{x,y} \}$ are independent random variables with continuous distribution, is of type $ \F $ (see \cite{DN2003,NNS2000}).
       Moreover, in $ d=2 $  the homogeneous ferromagnetic model  is of type $ \I $ (see \cite{NNS2000}).  
       In \cite{GNS2000}, an analysis of the zero-temperature stochastic Ising model on $ \Z^d $ with nearest-neighbour interactions distributed according to a measure $ \mu_{J}$ (disordered model) is performed. 
         In particular,  
       it is  proved that if the interactions are i.i.d. taking only the values $ \pm J $ then the two dimensional  model is of type $ \calM $.      
       An analogous result for $ d>2 $ with a temperature fast decreasing  to zero is obtained in \cite{CDS18}.      
       On the cubic lattice $ \Z^d $, if initial configuration is distributed according to a Bernoulli product measure with parameter $ p $  sufficiently close to $ 1 $ (i.e. if $ p>p^{\star}_{d} $), then the model is of type $ \F $, in particular each vertex fixates at the value $ +1 $  (see \cite{FSS2002}). Moreover in \cite{M2011} it is shown that $ p^{\star}_{d}\to 1/2 $ as $ d \to \infty $.
       For homogeneous trees of degree at least $ 3 $ and $ p $ sufficiently close to $ 1 $, it has been shown that  the model is of type $ \F $ (see \cite{CM2006,EN2015}).

    In \cite{DEKNS2016,DKNS2016}, the case in which one or infinitely many vertices are frozen is studied. The main result of the first paper is that for $d=2 $ the model, with infinitely  many  frozen vertices,  is of type $\F$. On the contrary, in the second paper the authors show that the model in $d=2 $ is  of 
type $\I $ when only one spin is frozen. 
       
       For articles on the stochastic Ising model on graphs other than  $ \Z^d $ see for example \cite{CNS2002, CDS18, CCK17, DKNS2013, DSM16, GNS2018, HJ2006}; in particular in \cite{CNS2002} it is shown that the zero-temperature Ising model on the hexagonal lattice is of type  $ \F $ and in \cite{CDS18} it is proved that it is not of type $ \F $ if  simultaneous spin flips are allowed.
       In \cite{GNS2018} the authors studied
       the Dilute Curie–Weiss Model, i.e. the Ising Model on a dense Erd\H{o}s-R\'enyi random graph, and proved that depending on the distribution of interactions there are different behaviors.

\medskip 

       In this paper,  we deal with  connected planar quasi-transitive graphs. 
        The quasi-transivity of the graph will be given by the invariance under translations and rotations. 
        We will show that, under mild 
         assumptions, the only rotations to consider are those of an angle
         $ \theta \in\left \{\frac{\pi}{3},\frac{\pi}{2},\frac{2}{3}\pi ,  \pi   \right \} $  (see Lemma \ref{l:rotations} and Theorem \ref{t:quasi-transitive}).
        Such a class of graphs includes, for instance, the square, the triangular and the hexagonal lattice. 

       Our first result on the zero-temperature stochastic Ising model (Theorem \ref{p:nec.cond.}) 
       shows that a necessary condition 
       for the model to be of type $ \I $ is that the underlying graph has the \emph{shrink property}. For example, the hexagonal lattice does not have the shrink property and the model on this lattice is of type $ \F $ (see \cite{NNS2000}). Thus, we will focus on a class of graphs having the shrink property. Actually, for technical reasons, we will use a potentially stronger definition of the shrink property that is the \emph{planar shrink property}.


       Our main result (Theorem \ref{t:main result}) shows that if $ p=1/2 $ and the  graph is invariant under rotations, translations and has  the planar shrink property, then the model is of type $ \I $. 
       
     Here we briefly present the general strategy to prove this achievement. First we show two preliminary results on general  attractive spin  systems with initial density $p \in (0,1]$ (see Lemmas \ref{l:order coupling}-\ref{l:it fixates from time 0}).  More precisely we show that, for an attractive system, 
if a spin fixates to $+1 $ with positive probability then the probability that it is constantly equal to $+1$  for all times $t \in [0,\infty)$ is positive. After the general  analysis developed in Theorems  1-2 
we  specifically study the  zero-temperature stochastic Ising model. 
       	First we show that, under the shrink property and the  translation-ergodicity,  the cardinality of any cluster grows to infinity  almost surely (Theorem \ref{pr:diverges}). By this preliminary result, we are able to show that the cluster at the origin  will intersect the boundary of any finite region infinitely often for  $ t \in [0, \infty) $ with probability one.
       	As already mentioned, we consider a planar graph that is invariant under translations and  rotations  of $ \theta \in\left \{\frac{\pi}{3},\frac{\pi}{2},\frac{2}{3}\pi \right \} $. Then, we construct  a planar regular region centered at the origin that  has  
       	the same rotation invariance of the graph.
       	By the FKG inequality and the rotation invariance of the region, the cluster    in the origin  will intersect all the sides of the regular region with a positive probability larger or equal to $p_{cross}$. We  stress that, for $t$   growing to infinity,  the quantity  $p_{cross}$  does not depend  on the size of the region. 
       	By these  properties and by the previous results 
       	we show that 
       	any ball centered in the origin has its spins equal to $ +1 $  infinitely often with a  probability larger of $p_{cross}$ (see Lemmas \ref{l:convex hull}-\ref{l:event D}).          
       	Thus,  with  probability at least $p_{cross}$ no site  will be able to fixate at the value $ -1 $.
       	Finally, by considering the initial density $ p= 1/2 $ and by Lemmas \ref{l:it fixates from time 0}-\ref{l:event D}, we show that  all sites flip infinitely often almost surely (see Theorem \ref{t:main result}).

      The plan of the paper is the following.  In Section \ref{s: preliminaries}, we define the zero-temperature stochastic Ising model, introduce the underlying graph
        and present some general  results on attractive systems that will be useful for our discussion (see Lemmas 1-3 and Theorem 1). 
        In Section \ref{s: main results}, the main result, Theorem \ref{t:main result}, is stated and proved through some lemmas and theorems.
      In Section \ref{s: examples}, we present an infinite class of graphs having the planar shrink property (see Theorem \ref{p: H planar shrink property}). 
       We also provide examples
of graphs that have and do not have the shrink property, cases where
the Ising model is of type $\I$ in the first case, and of type either $\mathcal{M}$ or
$\F$, in the latter.     
         
     \section{Preliminaries}
     \label{s: preliminaries}
    In this section, we introduce the graph underlying the zero-temperature stochastic Ising model and present some preliminary results that will be useful for our discussion in the next section.
    In particular, in Subsection \ref{s:preliminary attractive dynamics} we define the Markov process by the infinitesimal generator and by the Harris’  graphical representation. Moreover, we state two general lemmas (Lemma \ref{l:order coupling} and Lemma \ref{l:it fixates from time 0}) for Glauber attractive dynamics. 
     In Subsection \ref{s: graph introduction}, we introduce some notation on graphs and define the collection of graphs in which we
     	are interested. More precisely we present in Lemma 3 and Theorem 1 some properties of sets in $\R^2 $ that are translation and rotation  invariant. 
      In Subsection \ref{s: the model}, we describe in detail the zero-temperature stochastic Ising model $ I(G,p) $,
      where $G$ is the underlying graph and $p$ is the initial density. 
      Moreover, we prove Theorem \ref{p:nec.cond.}, which shows that the shrink property is a necessary condition 
       to obtain that the $ I(G,p) $-model is of type $\I$. 

    \subsection{Attractive spin systems}
      \label{s:preliminary attractive dynamics} 
     We now introduce the spin systems referring mainly to \cite[Chapter 3]{L1985}. 
     We consider a spin system $ (\sigma_t)_{t\geq 0} $, which describes $ \pm 1 $ spin flips dynamics on a countable set of vertices $ V $. 
     The state space is $ \Sigma=\{+1,-1\}^{V} $.     
     The value of the spin at vertex $ v\in V $ at time $ t$ will be denoted  by $ \sigma_t(v) $. We introduce the usual order relation $ \leq $ on $ \Sigma $: given two configurations $ \sigma,\sigma^\prime\in \Sigma $, we say that $ \sigma \leq  \sigma^\prime$ if for each $ v\in V $, $ \sigma(v) \leq  \sigma^\prime(v) $. The spin system $ (\sigma_t)_{t\geq 0} $ evolves as a Markov process on the state space $ \Sigma $ with infinitesimal generator $ L_{t} $, which acts on local functions $ f $, and  defined as
     \begin{equation}
     	\label{eq:generator}
     	(L_{t}f)(\sigma)=\sum_{v\in V}c_{t}(v,\sigma)\bigl(f(\sigma^v)-f(\sigma)\bigr),
     \end{equation} 
     where $ t\geq 0 $,  $ c_{t}(v,\sigma) $ is the flip rate of the spin at vertex $ v $, and $\sigma^v \in \Sigma$ is defined in the following way:
     \begin{equation*}
     	\sigma^v(u)=\begin{cases}
     		\sigma(u)	& \text{if $ u \neq v $}\\ 
     		-\sigma(u)	& \text{if $ u = v $}.
     	\end{cases}
     \end{equation*}
   We assume that  $c_{t}(v,\sigma) $ is a uniformly bounded non-negative function, which is  
 continuous on $ \sigma $ 
    and satisfies the condition
   \begin{equation}
   	\label{eq: condition on c}
   	\sup_{v\in V}\sum_{w\in V}\sup_{\sigma\in \Sigma}|c_{t}(v,\sigma)-c_{t}(v,\sigma^{w})|<\infty.
   \end{equation}
   The condition in \eqref{eq: condition on c} guarantees 
   the existence 
of the Markov process with infinitesimal generator $ L_{t} $ (see \cite{L1985}). We take the process $ (\sigma_t)_{t\geq 0} $  right continuous.

   We say that a spin system is \emph{attractive} if $ c_{t}(v,\sigma) $ is increasing in $ \sigma $ when $ \sigma (v)=-1 $ and decreasing in $\sigma$ when $ \sigma (v)=+1 $.
    We are in particular interested to study Glauber dynamics, for which  the  relation 
   \begin{equation}
   	\label{eq:c relation}
   	c_t(v,\sigma)=1-c_t(v,\sigma^v)
   \end{equation}
   holds for each $v\in V$, $\sigma\in \Sigma$ and $ t\geq 0 $. If the relation \eqref{eq:c relation} holds, then $ 0\leq c_t(v,\sigma) \leq 1 $.    
    	We write the flip rates in the form
    	\begin{equation}
    		\label{e: harris c condition}
    		c_t(v,\sigma)=\tilde{c}_{t}(v,(\sigma(u))_{u\in A_v}), 
    	\end{equation} 
    where $ A_v $ is a  subset of $ V $. 
     Under  \eqref{eq:c relation} and the assumption 
\begin{equation} \label{cond-A}
\sup_{v\in V}|A_v|<\infty,  
\end{equation}
 the process defined in \eqref{eq:generator} can be  constructed by the Harris’ graphical representation (see e.g. \cite{H1974,H1978, L2017, L1985}), which we now describe. 
   We consider a collection $(\calP_v)_{v\in V}$ of independent Poisson processes with rate $1$ interpreted as counting processes.
   For each $v\in V$, let $\calT_v=(\tau_{v,n}:n\in \N)$ be the ordered sequence of arrivals of the Poisson process $\calP_v$, associated with the vertex $ v $. The probability that there is a flip at vertex $v$ at time $ t $ (conditioning on the event $ \{t\in \calT_v\} $) is equal to $ c_{t}(v,\sigma_{t^{-}}) $, where $ \sigma_{t^{-}}:=\lim_{s\to t^{-}}\sigma_{s} $.
    For convenience, to describe these events in more detail, we can use a family of i.i.d. random variables $ (U_{v,n }:v\in V,n\in \N) $ distributed according to a uniform random variable in $ [0,1] $ and such that if $ U_{v,n }<c_{\tau_{v,n}}(v,\sigma_{\tau_{v,n}^{-}}) $, then the spin at $ v $ flips at time $ \tau_{v,n} $ (see \cite{H1978} and \cite{L1985}).  

    Lemma \ref{l:order coupling} below  is well known, but we present a proof in order to construct the coupling that will be used in the proof of  Lemma \ref{l:it fixates from time 0}.
   \begin{lemma}
   	\label{l:order coupling}
   	Given two  Glauber attractive dynamics $ (\sigma_t)_{t\geq 0} $ and $ (\sigma^{\prime}_t)_{t\geq 0} $ having the same generator and such that $ \sigma_0\leq \sigma^{\prime}_0 $,  there exists a coupling such that $ \sigma_t\leq \sigma^{\prime}_t  $ for each $ t\geq 0 $.		
   \end{lemma}
   \begin{proof}
   	By hypothesis the order relation is satisfied at the initial time. Hence, it is sufficient to consider a single arrival of the Poisson process, i.e. only a possible spin flip, in order to show that the order relation is maintained, i.e. we will show that $\sigma_{\tau_{v,n}^{-}}(v)=+1>\sigma^{\prime}_{\tau_{v,n}^{-}}(v)=-1 $  does not occur.

   	Let us explicitly construct the desired coulping.
   	We use  the same Poisson processes for the two systems, but different families of i.i.d. uniform random variables $ (U_{v,n }:v\in V,n\in \N) $ and  $ (U^{\prime}_{v,n }:v\in V,n\in \N) $ for  $(\sigma_t)_{t\geq 0} $ and $ (\sigma^{\prime}_t)_{t\geq 0}$ respectively.

 For each $ v, v' \in V $ and $ n\in \N $, let us consider the stopping time 
$\tau_{v,n }$. 
Moreover, for  $v' \in V \setminus \{v\} $ let us define 
$$
N(n,v, v') := \sup \{\ell \in \N:     \tau_{v',\ell } \leq \tau_{v,n } \} . 
$$
We notice that  $N(n,v, v) = n$, moreover for $v'\neq v$ 
one has $ \tau_{v',  N(n,v, v') }  <  \tau_{v,n }$ almost surely.

If $ \sigma_{\tau_{v,n}^{-}}(v)=\sigma^{\prime}_{\tau^{-}_{v,n}}(v) $ then we define $ U^{\prime}_{v,n }=U_{v,n } $, if  $\sigma_{\tau_{v,n}^{-}}(v)=-1<\sigma^{\prime}_{\tau_{v,n}^{-}}(v)=+1 $ then we define  $U^{\prime}_{v,n }=1-U_{v,n }$. 
Hence,  the random variable 
$ U^{\prime}_{v,n } $ is  a function of  
\begin{equation}\label{funzione}
(U_{v, 1},  \ldots , U_{v, n-1})
\end{equation} 
and of the  independent sequences  $( U_{v', \ell } : v' \in V \setminus \{v\}, \ell \leq N (n, v , v'))  $ and 
 $(\mathcal{P}_{v'} (t): v' \in V, t \leq  \tau_{v, n})$.

   	If  $ U^{\prime}_{v,n }=U_{v,n } $, by construction,  $ U^{\prime}_{v,n } $ is independent from $ U^{\prime}_{v,1 },\dots,U^{\prime}_{v,n-1 } $, which altogether are functions of $ U_{v,1 },\dots,U_{v,n-2 } $,  of $\{ U_{v', \ell } : v' \in V \setminus \{v \}, \ell \leq N(n-1 ,v, v') \} $ 
 and 
 $(\mathcal{P}_{v'} (t): v' \in V, t \leq  \tau_{v, n-1})$. Otherwise, if $U^{\prime}_{v,n }=1-U_{v,n }$ independence follows by:  	
 	\begin{multline*}
   			\mP(U^{\prime}_{v,n }\in [a,b]|U_{v,1 }=u_1,\dots, U_{v,n-1 }=u_{n-1})=\\
   			=\mP(U_{v,n }\in [1-b,1-a]|U_{v,1 }=u_1,\dots, U_{v,n-1 }=u_{n-1})=\\
   			=\mP(U_{v,n}\in [1-b,1-a])=b-a=\mP(U^{\prime}_{v,n }\in [a,b]),
   	\end{multline*}
   where $ 0<a<b<1 $ and $u_{1}, \ldots u_{n-1} \in (0,1)$. This implies that the distribution of $ U^{\prime}_{v,n } $ and the conditional distribution  of $ U^{\prime}_{v,n } $ given $ U^{\prime}_{v,1 },\dots,U^{\prime}_{v,n-1 } $ coincide, hence $ U^{\prime}_{v,n } $ is independent from  $ U^{\prime}_{v,1 },\dots,U^{\prime}_{v,n-1 } $ and $ U^{\prime}_{v,n } $ is a uniform random variable on $ [0,1] $. 
The independence of  different sequences  of uniform random variables can be proved in a similar way.

\medskip 

   	Whenever there is an arrival of a Poisson process, for example at time $ t $ for the vertex $ v $ (i.e. $ t\in \calT_v\ $), the following situations can arise:   

   	    \medskip
   	    
   		 Case $ \sigma_{\tau_{v,n}^{-}}(v)=\sigma^{\prime}_{\tau_{v,n}^{-}}(v)=-1 $.\\
   		Since $ U_{v,n }=U^{\prime}_{v,n } $ and $ c_{\tau_{v,n}}(v,\sigma_{\tau_{v,n}^{-}}) $ is increasing in $ \sigma $, we have the following three situations: if $ U_{v,n }<c_{\tau_{v,n}}(v,\sigma_{\tau_{v,n}^{-}})\leq c_{\tau_{v,n}}(v,\sigma^\prime_{\tau_{v,n}^{-}})  $ then both systems change the spin value at $ v $; if $c_{\tau_{v,n}}(v,\sigma_{\tau_{v,n}^{-}})\leq  U_{v,n }< c_{\tau_{v,n}}(v,\sigma^\prime_{\tau_{v,n}^{-}})  $ then in the system $ (\sigma_t)_{t\geq 0} $ the spin at $ v $ does not change its value, while in  $ (\sigma^{\prime}_t)_{t\geq 0} $ the spin flip  occurs at $ v $; if $c_{\tau_{v,n}}(v,\sigma_{\tau_{v,n}^{-}})\leq  c_{\tau_{v,n}}(v,\sigma^\prime_{\tau_{v,n}^{-}})\leq  U_{v,n } $ then both systems do not have the spin flip at $ v $. 
   		 In all  these three situations the order relation is maintained.
   		 \medskip
  		
   	    Case $ \sigma_{\tau_{v,n}^{-}}(v)=\sigma^{\prime}_{\tau_{v,n}^{-}}(v)=+1 $.\\
   		Since $ U_{v,n }=U^{\prime}_{v,n } $ and $ c_{\tau_{v,n}}(v,\sigma_{\tau_{v,n}^{-}}) $ is decreasing in $ \sigma $, we have the following three situations: if $ U_{v,n }<c_{\tau_{v,n}}(v,\sigma^\prime_{\tau_{v,n}^{-}})\leq c_{\tau_{v,n}}(v,\sigma_{\tau_{v,n}^{-}})  $ then both systems change the spin value at $ v $; if $c_{\tau_{v,n}}(v,\sigma^\prime_{\tau_{v,n}^{-}})\leq  U_{v,n }< c_{\tau_{v,n}}(v,\sigma_{\tau_{v,n}^{-}})  $ then in the system $ (\sigma_t)_{t\geq 0} $ the spin at $ v $ change its value, while in  $ (\sigma^{\prime}_t)_{t\geq 0} $ the spin flip does not occur at $ v $; if $c_{\tau_{v,n}}(v,\sigma^\prime_{\tau_{v,n}^{-}})\leq  c_{\tau_{v,n}}(v,\sigma_{\tau_{v,n}^{-}})\leq  U_{v,n } $ then both systems do not have the spin flip at $ v $.  In all  these three situations the order relation is maintained.
        \medskip
        
       Case $\sigma_{\tau_{v,n}^{-}}(v)=-1<\sigma^{\prime}_{\tau_{v,n}^{-}}(v)=+1 $.\\
   		If $ U_{v,n }<c_{\tau_{v,n}}(v,\sigma_{\tau_{v,n}^{-}}) $ then, since in this case  $ U^{\prime}_{v,n }=1- U_{v,n } $, by using the relation \eqref{eq:c relation} for Glauber  dynamics and by attractivity, one has that
   		\begin{equation*}
   			 U^{\prime}_{v,n }=1- U_{v,n }>1-c_{\tau_{v,n}}(v,\sigma_{\tau_{v,n}^{-}})=c_{\tau_{v,n}}(v,\sigma^v_{\tau_{v,n}^{-}})\geq c_{\tau_{v,n}}(v,\sigma^\prime_{\tau_{v,n}^{-}}). 
   		\end{equation*}
   		 Thus, in the system $ (\sigma_t)_{t\geq 0} $ the spin at $ v $ changes its value, while in  $ (\sigma^{\prime}_t)_{t\geq 0} $ the spin flip does not occur at $ v $, maintaining the order relation. If $ U_{v,n }\geq c_{\tau_{v,n}}(v,\sigma_{\tau_{v,n}^{-}}) $, the spin at $v$ in  $ (\sigma_t)_{t\geq 0} $ does not change and therefore the order relation is maintained.
   		\medskip

     	By previous cases  and since $ \sigma_0\leq \sigma^{\prime}_0 $, one deduces that the order relation is maintained at any time. Hence $\sigma_{\tau_{v,n}^{-}}(v)=+1$ and $ \sigma^{\prime}_{\tau_{v,n}^{-}}(v)=-1  $	does not occur.		
   \end{proof} 
      Now, we give the following definition
 which will be used in the next Lemma 
\ref{l:it fixates from time 0}.     
      \begin{definition}
      We say that a vertex $ v $ \emph{fixates} if the spin at $ v $ flips only finitely many times and we say that a vertex \emph{fixates from time zero} if its spin never flips.
      \end{definition}
  In the following Lemma \ref{l:it fixates from time 0}, for Glauber attractive systems,  we 
 compare the probability that a spin fixates or fixates 
from time zero.
%
%
   \begin{lemma}
   	\label{l:it fixates from time 0}
   	 	Consider a Glauber attractive dynamics $ (\sigma_t)_{t\geq 0} $ where $ \sigma_0 $ has density $ p\in(0,1] $. 
   	 If  a vertex $w$ fixates at the value  $ +1 $ with positive probability, then  the vertex	$ w $ fixates at the value  $ +1 $ from time zero with positive probability.
   \end{lemma}
   \begin{proof}
   	We define $ T_{w}:=\inf\{s\geq 0:\sigma_t(w)=+1 \ \ \forall t\geq s\} $, where $ \inf\emptyset=+\infty $.
   		We assume that $ \mP(T_w<\infty )=\rho>0 $ and we choose   $ \bar{t} $ such that $ \mP(T_w<\bar{t})\geq \rho/2 $. 

   	We consider a spin system $(\sigma_t)_{t\geq 0} $, described through the Harris' graphical representation with the independent Poisson processes of rate $ 1 $ $(\calP_v:v\in V)$  and the i.i.d. uniform random variables $ (U_{v,n }:v\in V,n\in \N) $, with initial configuration $ \sigma_0 $ distributed according to a Bernoulli product measure with parameter $ p $. 
   	Now, we construct another system $(\sigma^{\prime}_t)_{t\geq 0}$ with the same distribution. We make a resampling (indipendently by all other random variables already introduced) of the spin at vertex $ w $ in the initial configuration, such that
   	\begin{equation*}
   		\sigma^{\prime}_0(w)=\begin{cases}
   			+1	& \text{with probability $p$}\\ 
   			-1	& \text{with probability $1-p$},
   		\end{cases}
   	\end{equation*}
   	and  $\sigma^{\prime}_0(u)=\sigma_0(u) $, for each $  u \neq w $.

   	We define a new Poisson process $\calP_w^{\prime}$ that after time $\bar{t}$ has  the same arrivals of $\calP_w$.
   	In the interval $[0,\bar{t}]$, $\calP_w^{\prime}$ is a Poisson process of rate $ 1 $ indipendent by $\calP_w$.
   	This new process, by independent increments property, is still a Poisson process of rate $ 1 $. 
   	With positive probability one has $ \calP^{\prime}_w(\bar{t})=0 $. 
   	Thus, by independence, 
   	with probability at least $ \frac{\rho}{2}pe^{-\bar{t}} $, the following three events occur:
   	\begin{equation}
   		\label{e: 3 events}
   		\{ T_w <\bar{t}\}, \quad \{\sigma^{\prime}_0(w)=+1\}, \quad \{\calP^{\prime}_w(\bar{t})=0\}. 
   	\end{equation} 
   	Whenever these three independent events occur, we define $ (U_{v,n }^{\prime}:v\in V,n\in \N) $ as follows:  
   	\begin{itemize}
   		\item for $ v \neq w $, 
   		  $U^{\prime}_{v,n }= U_{v,n } $ when $ \sigma_{\tau^{-}_{v,n}}(v)=\sigma^{\prime}_{\tau^{-}_{v,n}}(v) $, otherwise  $U^{\prime}_{v,n }=1-U_{v,n }$;
   		\item for $ v = w $, $ U_{w,n }^{\prime}=U_{w,n+\calP_w(\bar{t})} $ 
 when $\sigma_{\tau^{-}_{w,n+\calP_w(\bar{t})}}(w)=\sigma^{\prime}_{\tau^{-}_{w,n}}(w) $, otherwise $ U_{w,n }^{\prime}=1-U_{w,n+\calP_w(\bar{t})} $.
   	\end{itemize}
   	 If one of the three events in \eqref{e: 3 events} does not occur, the uniform random variables $ (U_{v,n }^{\prime}:v\in V,n\in \N) $  will be defined as $U^{\prime}_{v,n }=U_{v,n }$ for each $ v\in V $ and $ n\in \N  $.

   	Now, we suppose that the three events in \eqref{e: 3 events} occur. By construction, we have that $ \sigma_0\leq \sigma^{\prime}_0 $. 
   	We show that for  $ t \in [0,\bar{t}]  $, one has $ \sigma_t\leq \sigma^{\prime}_t $.
   	 Since $ \calP^{\prime}_w(\bar{t})=0 $, 
   	  then in the process  $ (\sigma^{\prime}_t)_{t\geq 0} $ the spin  at $ w $ remains equal to $ +1 $ until time $\bar{t}$; hence $ \sigma_t(w)\leq \sigma^{\prime}_t(w) $ for all $ t\leq \bar{t} $.
   	When there is an arrival of a Poisson process $ \calP_v $ with $ v\neq w $, by using the same coupling of Lemma \ref{l:order coupling}, it follows that the desired  order relation is maintained until time $ \bar{t} $.  
   	In particular, at time $\bar t$,  $ \sigma_{\bar{t}} \leq \sigma^{\prime}_{\bar{t}} $. Now, applying the result of Lemma \ref{l:order coupling} by considering $ \bar{t} $ as  initial time, it follows that  $ \sigma_t\leq \sigma^{\prime}_t  $ for each $ t\geq 0 $. Hence, the vertex $ w $ fixates from time zero with probability at least  $ \frac{\rho}{2}pe^{-\bar{t}} > 0$, concluding
   		the proof.  
   \end{proof}
    \subsection{Notations and basic properties of graphs}
    \label{s: graph introduction}
       We begin this subsection by presenting Lemma \ref{l:rotations} that applies to  subsets of $ \R^2 $ which are invariant by translations and rotations. with the purpose of applying it
    to connected planar  infinite quasi-transitive graphs.
    Later in the subsection, we introduce some definitions and notation on the graphs (see e.g. \cite{RD-graph} and \cite{HJ2006}) 
    and we present the graphs on which the zero-temperature stochastic Ising model will be constructed. 

\medskip 

    Let us denote by $ \|\cdot\| $ the Euclidean norm   and by $ B(x,r) $ the ball of radius $ r>0 $ centered in $ x $. 
    For any $ S\subset \R^2 $ and  $ \bar{x} \in \R^2$, we define 
 \emph{the translation of a set as}
    \begin{equation*}
    	S+\bar{x}:=\{x+\bar{x}:x\in S\}.
    \end{equation*}
     Given $\theta \in (0, \pi]$, we say that $ S $ is invariant under rotation of $ \theta $ if there exists a point $ O\in \R^2 $, which we assume to be the origin, such that $  R_{\theta} (S) \subset S $, where $R_{\theta} $ is the  rotation in the plane with center $ O $ and angle $\theta$.

    \begin{lemma}
    	\label{l:rotations}
    	Let  $ \bar{x} \in \R^2$  be a non-zero vector and    $ R_{\theta} $ a rotation in the plane
    	 with center $ O $ and angle $\theta \in (0, \pi]$. 
    	If $ S\subset \R^2 $ is a non-empty set such that
    	\begin{itemize}
    		\item $S$ has a finite number of points in any ball;
    		\item $S +\bar{x}  \subset S $;
    		\item $ R_{\theta}(S) \subset S $.
    	\end{itemize}
    	Then  $S +\bar{x}= S $, $ R_{\theta}(S)= S $ and $ \theta \in\left \{\frac{\pi}{3},\frac{\pi}{2},\frac{2}{3}\pi ,  \pi   \right \} $.
    \end{lemma}
    \begin{proof}
    	By hypothesis $ S $ is non-empty. Thus, by $S +\bar{x}  \subset S $, there exists a  point $ v \in S$, with $v\neq O $.   Regarding the rotation, we  write $ \theta =2\pi \alpha $ where $ \alpha \in \R$. Now, if $ \alpha \in \R \setminus \Q $ the set $ \{R^n_{\theta}(v):n\in \N \}\subset S $ is  dense in $ \|v\|S^1 $
    	that contradicts the property that each ball in $ \R^2 $ contains a finite number of points of $ S $. Hence, $ \alpha$ 
    	necessarily belongs to $ \Q $.

    	Let $ \alpha\in \Q $,   we  write  $ \alpha=\frac{m}{n} $ with $ m,n\in \N $ coprime. By B\'ezout's lemma, there exist $ a,b\in \Z $ such that $ am+bn=1 $. Let us select an integer $k \in \N$ such that $a +k n \in \N$. One has 
    	$ R^{a+kn}_{\theta}=R_{\frac{2\pi}{n}-2\pi (b-km)}=R_{\frac{2\pi}{n}} $. Thus $R_{\frac{2\pi}{n}}(S) \subset S$. 
    	Hence,    we can consider  only the angles of the form $\theta =2 \pi \frac{1}{n}$, for $n \in \N$.

    	By applying $(n-1)$ times the rotation $R_{\frac{2\pi}{n}}$ one obtains $R_{-\frac{2\pi}{n}}(S)\subset S$, therefore the rotations with rational  $ \alpha $ are surjective onto $S$ and consequently also invertible on $S$. 
    	In particular, $R_{ 2 \pi \alpha }(S) =S $ for any $\alpha \in \Q$. \medskip

    	Now we define 
    	$$
    	r:=\min\{||w||: w \in \R^2, \quad S +w \subset S   \}   
    	$$ 
        that is well-posed  because, by hypothesis, there exists $ \bar{x} \in \R^2$ such that $S +\bar{x}  \subset S $ and $S $ has a finite number of points in any ball.

    	Therefore there exists  $ \bar{v}_0 \in \R^2 $  such that $S + \bar{v}_0 \subset S $ with $||\bar{v}_0||=r$. Notice that, without loss of generality, one can assume that $ r=1 $ and $ \bar{v}_0=(1,0) $. 
    	
    	\noindent Let us observe  that, by $R_{ 2 \pi \frac{1}{n} }(S) =S $, it follows 
    	$$
    	S +  \bar{v}_k  \subset S , 
    	$$ 
    	where $\bar{v}_k = \bigl(\cos(\frac{2\pi k}{n}),\sin(\frac{2\pi k }{n})\bigr) $ for $k =0, \ldots , n-1$. 
    	From the fact that 
    	\begin{equation*}
    		-\bar{v}_0=\sum_{k=1}^{n-1}\bar{v}_k, 
    	\end{equation*}	
    	one has  that $ S -\bar{v}_0 \subset S$. 
    	Therefore the translation with respect to  $\bar{v}_0$ is 
    	surjective onto $S$ and being also injective it is invertible on $S$. 
    	Therefore $S\pm\bar{v}_k =S $ for $k =0, \ldots , n-1$. 
    	Hence,    one has
    	$$
    	S +\bar{v}_1  -\bar{v}_0=S. 
    	$$ 
    	The norm of $\bar{v}_1-\bar{v}_0$ is   $\sqrt{2-2\cos (\frac{2\pi }{n})}$.  Since $ \bar{v}_0$ is a minimal norm vector such that $S +\bar{v}_0 =S$, one has
    	\begin{equation}\label{mag1}
    		\|\bar{v}_1-\bar{v}_0\|=\sqrt{2-2\cos \left (\frac{2\pi }{n} \right )}                \geq 1.
    	\end{equation}
    	By \eqref{mag1} 
    	one obtains that $ \cos\bigl(\frac{2\pi}{n}\bigr)\leq \frac{1}{2} $, which gives  $ n\in \{2,3,4,5,6\}$. 
    	
    	In order to get the result, we need to show that $n =5 $ contradicts $r =1$. In this regard, we consider  
    	$$ 
    	S+\bar{v}_0+ \bar{v}_2 = S 
    	$$ where 
    	$ \bar{v}_2=(\cos(4 \pi /5),\sin(4 \pi/5)) $. Notice that  $ \|\bar{v}_0+\bar{v}_2\|=\sqrt{2+2\cos\bigl(\frac{4}{5}\pi\bigr)}<1 $, which
    	 contradicts $r =1$. 
    	Therefore $ \theta \in \big\{\frac{\pi}{3},\frac{\pi}{2},\frac{2}{3}\pi ,  \pi \big\} $  and this concludes the proof.	   	
    \end{proof}   
    \begin{remark}
    	We note that a set $ S $ satisfying the hypotheses of Lemma \ref{l:rotations} must be  countable. 
 Notice that for  $\bar x =(0,1)$ and $ \theta \in\left \{\frac{\pi}{3},\frac{\pi}{2},\frac{2}{3}\pi ,  \pi   \right \} $ there exist examples of sets $S$ such that $S + \bar x= S$ and 
$R_{\theta } (S) =S$.
 Moreover, there are examples where  $R_{2\pi/3 } (S) =S$ but  $R_{\pi/3 } (S) \neq S$ and examples such that 
 $R_{\pi } (S) =S$ but  $R_{\pi/2 } (S) \neq S$.   
    \end{remark}
    
    Let us now recall some definitions and notation 
     on graph theory (see e.g. \cite{RD-graph,HJ2006}).

    Let $ G=(V,E) $ be a graph, where  $ V $ is the set of its vertices (or sites)  and $ E\subset V\times V $ is the set of its edges.    
    The \emph{degree} of a vertex $ v\in V $, denoted by $ deg(v) $, is the number of neighbours of $ v $, i.e. $ deg(v):=|\bigl\{u\in V:\{u,v\}\in E\bigr\}| $. The \emph{maximum degree} of $ G $ is $ \Delta(G):=\sup\{deg(v):v\in V\} $.
    Given $ v\in V $ and $ S \subset V $, we denote by $ N_S(v) $ the set of neighbours of $ v $ in $ S $, i.e. $ N_S(v):=\bigl\{u\in S:\{u,v\}\in E\big\} $ and we define the degree of $ v $ in $ S $ as $ deg_S(v):=|N_S(v)| $. 
    Given $ U\subset V $, the induced subgraph $ G[U] $ is the graph whose vertex set is $ U $ and whose edge set consists precisely of the edges $ \{u,v\}\in E $  with  $ u,v\in U $.
    A \emph{path}  connecting a vertex $ v $ to a vertex $ u $ is a non-empty graph $ P=(V(P),E(P)) $, where $ V(P)=\{v_0=v,v_1,\dots,v_{m-1},v_m=u\} $, the vertices $ v_i $ are all distinct and $ E(P)=\{\{v_i,v_{i+1}\}\in E: i=0,\dots,m-1\} $; 
     $ m $ is the length of the path $ P $. If $ P=(V(P),E(P)) $ is a path connecting $ v $ to $ u $, then the graph $ C:=(V(P),E(P)\cup \{\{u,v\}\}) $ 
     is called a \emph{cycle}.  
    A graph $ G=(V,E) $ is said to be \emph{connected} if for any two  vertices $ u,v\in V $ there exists a path connecting them. We say that $ U\subset V $   is  connected if the induced subgraph $ G[U] $ is connected.   
    We denote by $ d_G(u,v) $ the distance in $ G $ of two vertices $ u $ and $ v $ defined as the length of a shortest path connecting $ u $ to $ v $.  
    Given a subset $ U\subset V $, we define the \emph{external boundary} of $ U $ as the set $ \partial_{ext} U:=\{v \in V\setminus U: \exists u\in U \ \text{s.t.} \  \{v,u\}\in E\} $.
    Now, we provide the following definitions (see e.g.   \cite{HJ2006}). 
    \begin{definition}[Graph automorphism]
    	Let $ G=(V,E) $ be a graph. A bijective map $ \phi \colon V\to V $ is said to be a \emph{graph automorphism} if $ \{u,v\}\in E \iff \{\phi(u),\phi(v)\}\in E $.
    \end{definition}
    \begin{definition}[Transitive graph]
    	\label{d: transitive graph}
    	A graph $ G=(V,E) $ is called \emph{transitive} if for any $ u,v \in V $ there is a graph automorphism mapping $ u $ on $ v $.
    \end{definition}
    \begin{definition}[Quasi-transitive graph]
    	\label{d: quasi-transitive graph}
    	A graph $ G=(V,E) $ is said to be \emph{quasi-transitive}
    	if $ V $ can be partitioned into a finite number of vertex
    	sets $ V_1,\dots, V_N $ such that for any $ i=1,\dots,N $ and any $ u,v \in V_i $, there exists a
    	graph automorphism  mapping $ u $ on $ v $. 
    \end{definition}
    Now we introduce planar graphs, which play a central role in our paper. 
    An \emph{arc} is a subset of $ \R^2 $ that is the union of finitely many segments and is homeomorphic to the closed interval $ [0,1] $. The images of $ 0 $ and $ 1 $ under such  a homeomorphism are the \emph{endpoints} of this arc. If $ A $ is an arc with endpoints $ x $ and $ y $, the \emph{interior} of $ A $ is $ A\setminus\{x,y\} $ (see \cite{RD-graph}). 
    
    A \emph{plane graph} is a pair $ G=(V,E) $  that satisfies the following properties:  
    \begin{enumerate}
    	\item $ V\subset \R^2 $ is at most countable;
    	\item every edge is an arc between two vertices;
    	\item different edges have different sets of endpoints;
    	\item the interior of an edge contains no vertex and no point of any other edge.
    \end{enumerate}  
    A graph $ G=(V,E) $ is said to be \emph{planar} if it can be embedded in the plane, i.e. it is isomorphic to a plane graph $ \tilde{G} $. The plane graph $\tilde{G}$ is called a \emph{drawing} of $ G $ or \emph{embedding} of $ G $ in the plane $ \R^2 $. We can identify a planar graph with its embedding in $ \R^2 $. Similarly, we say that $ G=(V,E) $ is embedded in $ \R^d $ if $ V\subset \R^d $ is at most countable and (2)-(4) hold (see e.g. \cite{DSM16, RD-graph}). 

Given a plane graph $ G=(V,E) $ and  a set $ S\subset V $, let $ Conv_G(S) :=Conv(S) \cap V $, where 
$Conv(S)$ denotes the convex hull of $S$. 
   We now  define the shrink and planar shrink property, which will be central to our future discussion.  
    \begin{definition}[Shrink property]
    	Given a graph $ G=(V,E) $, we say that $ G $ has the \emph{shrink property} if for each  subset $ S\subset V $ with finite cardinality, there exists $ u\in S $ such that $ deg_{V\setminus S}(u)\geq deg_S(u) $.
    \end{definition}
    Given 
    a line $\ell \in \mathbb{R}^2 $, we denote by $ H^{\ell}_1\subset \R^2 $ and $ H^{\ell}_2\subset \R^2 $ the closed half-planes having $\ell$ as boundary. Given a non-empty  subset $ S\subset V $ and a line $ \ell $,  we define 
    $  S_1^{\ell}=S \cap H^{\ell}_1 $ and $  S_2^{\ell}=S\cap H^{\ell}_2 $;  clearly  $ S=S_1^{\ell}\cup S_2^{\ell} $. 
    \begin{definition}[Planar shrink property]
    	\label{d:planar shrink property}
    	For a plane graph $ G=(V,E) $, we say that $ G $ has the \emph{planar shrink property} when, for any non-empty set $ S\subset V $ and for any line $ \ell $, one has: 
\begin{itemize}
\item[]   for   $i = 1,2 $, if   $S_i^{\ell}$ is non-empty and has finite cardinality,  then  there exists
    	$ u\in S_i^{\ell} $ such that $ deg_{V\setminus S}(u)\geq deg_{S}(u) $. 
\end{itemize}
We say that a planar graph $ G $ has the \emph{planar shrink property} if there exists an embedding of $ G $ in the plane for which such a property holds.
    \end{definition}
    For a plane graph, it is immediate to note that the planar shrink property implies the shrink property.

    We are interested in a connected planar 
     infinite graph $ G=(V,E) $ with a specified embedding in $ \R^2 $  such that the following conditions hold:
    \begin{description}
    	\item[(C1)] 
    	There exists a non-zero vector $ \bar{x} $ such that $ V+\bar{x} \subset V$ 
    	and for any $ u,v\in V $, 
    	\begin{equation*}
    		\{u,v\}\in E \iff \{u+\bar{x},v+\bar{x}\}\in E. 
    	\end{equation*}
    Then	we say that $G$ is  translation invariant  with respect to the vector $ \bar{x} $. 
    	\item[(C2)] There exists a point $ O\in \R^2 $ 
and $\theta \in (0, \pi ]$ such that $ R_{\theta}(V)\subset V $ and for any $ u,v\in V $,
    	\begin{equation*}
    		\{u,v\}\in E\iff \{ R_{\theta}(u), R_{\theta}(v)\}\in E.
    	\end{equation*}
Then	we say that $G$ is  rotation invariant  with respect to $R_{\theta} $.
    	\item[(C3)] Each ball in $ \R^2 $ contains a finite number of vertices of $ G $. 
    \end{description}
    By Lemma \ref{l:rotations},  it follows that a graph satisfying   conditions (C1), (C2) and (C3) has 	$ \theta \in\left \{\frac{\pi}{3},\frac{\pi}{2},\frac{2}{3}\pi ,  \pi   \right \} $ and  the translations and rotations in (C1), (C2) are   graph automorphisms. For reasons that will become clear in the following, we do not deal with $ \theta=\pi $.
The translations and rotations 
 provide a partition of $V$ in classes, in any case 
 the partition given in Definitions \ref{d: transitive graph} and \ref{d: quasi-transitive graph} can be  finer than the one 
given by only translations and rotations.  
 For $\theta = \pi$, it is straightforward to exhibit an example  where the number of classes  is infinite. For instance, we can consider $ G=(V,E) $ where $ V=\Z^2 $ and  the edge set is
 	\begin{equation*}
 		E=\bigl\{\{(i,j),(i,j+1)\}:i,j\in \Z\bigr\}\cup \bigl\{\{(i,0),(i+1,0)\}:i\in \Z\bigr\}.
 	\end{equation*} 
  It is immediate to note that $ G $ is invariant under translation with respect to the vector $ (1,0) $ and is invariant under rotation of $ \pi $ in the origin but not of $ \pi/2 $. Moreover, the classes of $ G $ are $ C_n=\{(i,\pm n):i\in \Z\} $ for $ n\in \N_0 $ (see Figure \ref{f: ex180} in Section \ref{s: examples}).
  We  present the following result for plane graphs that  are translation and rotation invariant.

  \begin{theorem}
    	\label{t:quasi-transitive}
    	If $ G $ is a plane graph  satisfying the conditions (C1), (C2) and (C3), then  $ \theta \in\left \{\frac{\pi}{3},\frac{\pi}{2},\frac{2}{3}\pi,\pi   \right \} $. Moreover if    $ \theta \in\left \{\frac{\pi}{3},\frac{\pi}{2},\frac{2}{3}\pi   \right \} $ then the plane graph $ G $ is either transitive or quasi-transive. 
    \end{theorem} 
\begin{proof}
	The first part of the statement has already been discussed above.
 It is sufficient to prove that for $ \theta \in\left \{\frac{\pi}{3},\frac{\pi}{2},\frac{2}{3}\pi   \right \} $ 
the number of classes is finite. 
 Let  $ \theta \in\left \{\frac{\pi}{3},\frac{\pi}{2},\frac{2}{3}\pi   \right \} $, 
 the plane graph  $ G $ is translation invariant with respect to the linear independent vectors $ \bar{x} $ and $\bar{y}:= R_{\theta}\bar{x}  $. The number of classes is at most the number of vertices belonging to the closed  parallelogram spanned by vectors $ \bar{x} $ and $ \bar{y} $. 
By  (C3)  follows that the  number of vertices in this parallelogram is finite. 
\end{proof}
  
      \begin{figure}
    	\centering
    	\includegraphics[scale=0.41]{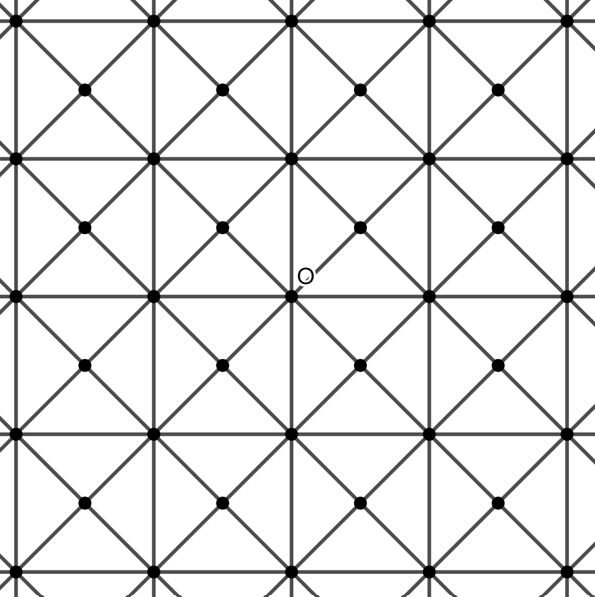} \quad
    	\includegraphics[scale=0.40]{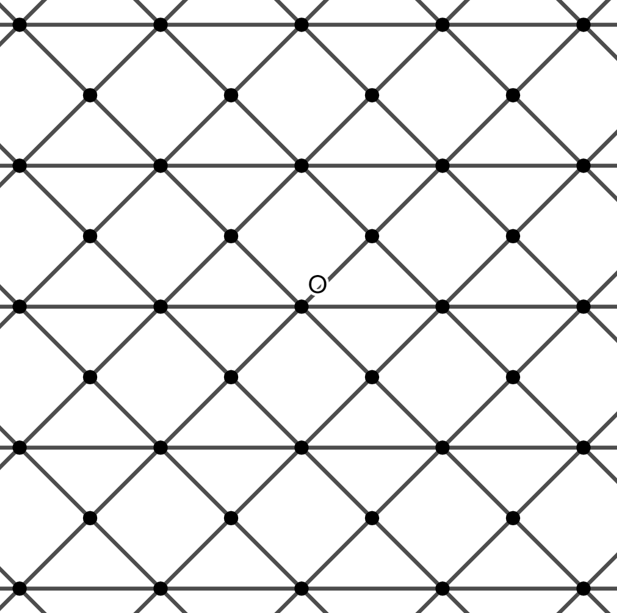} 
    	\caption{The graph on the left belongs to $ \G(4) $, hence it is invariant under rotation of $ \pi/2 $. The graph on the right is only invariant under rotation of $ \pi $.
    	 }
    	\label{f: pi/2 e pi}
    \end{figure}
    
    Now, we introduce the class of plane graphs  $ \G(a) $ with $ a\in \{3,4,6\} $ that is  the collection of connected infinite graphs (with  finite maximal degree) satisfying  conditions (C1)-(C3) with $ \theta=\theta(a)=2\pi/a $. It is immediate to notice that $ \G(6)\subset \G(3) $.
    Let, furthermore,  $ \G:=\G(3)\cup \G(4) $. 
    In Section \ref{s: examples}, we will provide various examples of such graphs.

   \subsection{The $ \mathbf{I(G,p)} $-model }
   \label{s: the model}
   	We consider the stochastic process $ (\sigma_t)_{t\geq 0} $, which describes $ \pm 1 $ spin flips dynamics on an infinite graph  $ G=(V,E) $ with $ \Delta(G)<\infty $. 
   	 The state space is $ \Sigma=\{+1,-1\}^{V} $ and the initial state is distributed according to a Bernoulli product measure with density $ p \in [0,1]$ of spins $ +1 $ and $ 1-p $ of spins $ -1 $. The  process corresponds to the zero-temperature limit of Glauber dynamics for an Ising model with formal Hamiltonian
   \begin{equation}
   	\label{eq:hamiltonian}
   	\mathcal{H}(\sigma)=-\sum_{\substack{u,v \in V:\\ \{u,v\}\in E}}\sigma(u)\sigma(v), 
   \end{equation}
   where $ \sigma \in \Sigma $.
   The definition \eqref{eq:hamiltonian} is not well posed for infinite graphs. For this reason, we introduce the changes in energy at vertex $ v\in V $ as
   \begin{equation*}
   	\Delta\mathcal{H}_v(\sigma)=2\sum_{\substack{u \in V:\\ \{u,v\} \in E}}\sigma(u)\sigma(v).
   \end{equation*}

   The process $ (\sigma_t)_{t\geq 0} $ is a Markov process on $ \Sigma $ with infinitesimal generator having as flip rates 
   \begin{equation}
   	\label{e:our rates}
   	c_{t}(v,\sigma)=\begin{cases}
   		0 & \text{if $ \Delta\mathcal{H}_v(\sigma)>0 $} \\
   		\frac{1}{2} & \text{if $ \Delta\mathcal{H}_v(\sigma)=0$ }\\
   		1 & \text{if $ \Delta\mathcal{H}_v(\sigma)<0 $}.
   	\end{cases}
   \end{equation}

   It is immediate to notice that the stochastic process is well defined, indeed the supremum in \eqref{eq: condition on c} is bounded by $ \Delta(G) $ 
    (see \cite{L1985}).
     We note that the process is a Glauber attractive dynamics.
       Furthermore, since $ \Delta(G)<\infty $, the flip rates in \eqref{e:our rates} satisfy the condition in \eqref{cond-A} with $ A_v=N_V(v) $. Therefore, this process can be  constructed by the Harris’ graphical representation (see \cite{H1974,H1978, L2017, L1985}).
   In the following, we will refer to this model 
   as  $ I(G,p) $-model where $G$ is the underlying graph and $p$ is the  density of the Bernoulli product measure for the initial configuration. Now, in Theorem \ref{p:nec.cond.}, we show that the shrink property is a necessary condition 
 to obtain that the $ I(G,p) $-model is of type $\I$.

   \begin{theorem}
   	\label{p:nec.cond.}
   	Let $ G=(V,E) $ be a graph with $ \Delta(G)<\infty $. 
   	 If $ G $ does not have the shrink property then for any $ p\in[0,1] $ the $ I(G,p) $-model is not of type $ \I $.
   \end{theorem} 
   \begin{proof}
   First let us consider the case $ p\in (0,1]$. Since  $ G $ does not have the shrink property then there exists a finite subset $ S\subset V $  such that $ deg_{V\setminus S}(u)< deg_S(u) $  for any $ u\in S $. Moreover, one has 
   \begin{equation}
   	\label{e: lb S}
   	\mP\biggl( \bigcap_{ u \in S } \{ \sigma_0(u)=+1 \} \biggr)=p^{|S|}>0.
   \end{equation}
    Since $ deg_{V\setminus S}(u)< deg_S(u) $  for every $ u\in S $, no site in $ S $ can change the value of its spin if $ \sigma_0(u)=+1 $ for each $ u\in S $. 

 This fact implies that 
   \begin{equation*}
   	\mP\bigl(\text{each vertex in $ S $ fixates at the value $ +1 $ from time zero}\bigr)\geq p^{|S|}>0.
   \end{equation*}
    Thus,  for any $ p\in(0,1] $ the $ I(G,p) $-model is not of type $ \I $.

    Let us now consider  $ p=0$. In this case,  all sites fixate from time $ 0 $ at the value $ -1 $ almost surely and the $ I(G,0) $-model is of type $ \F $. 
   	\end{proof}
    \begin{remark}
    	Note that Theorem \ref{p:nec.cond.} does not imply that the $ I(G,p) $-model is of type $ \F $ or $ \calM $. In fact if 
the graph $G$ is not invariant under translation it could happen   that the model does not belong to any of the three classes $\F, \, \calM$ and $\I$.
    \end{remark}
   
    If $ G $ does not have the shrink property, it is possible to show examples in which the $ I(G,p) $-model is respectively of type $ \F $ or $ \calM $. It is known (see \cite{CNS2002,NNS2000}) that if $ G $ is the hexagonal lattice then the $ I(G,p) $-model is of type $ \F $. Now, we show in Example \ref{e:lem.ex.M} that if we consider $ G\in \G(4) $ as in Figure \ref{f: exampleM} then for any $ p\in (0,1) $ the  $ I(G,p) $-model is of type $ \calM $.
    
    \begin{figure}[htp]
    	\centering
    	\includegraphics[scale=0.50]{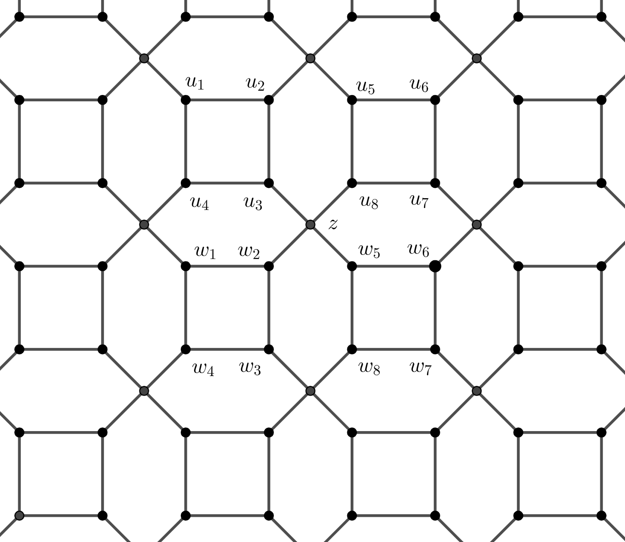}
    	\caption{Example of a graph $ G\in \G(4) $ that does not have the shrink property and, for  $ p\in (0,1) $, the $ I(G,p) $-model is of type $ \calM $.  
    	}
    	\label{f: exampleM}
    \end{figure}  
    
    \begin{example}
    	\label{e:lem.ex.M}
    	Let $ G\in \G $ be the graph in Figure \ref{f: exampleM}. For any $ p\in (0,1) $ the  $ I(G,p) $-model is of type $ \calM $.
    	Let $ u_i\in V $ with $ deg(u_i)=3 $ for $ i=1,\dots,8 $ as in Figure \ref{f: exampleM}. Let $ S=\{u_1,\dots,u_4\} $. Since $ deg_{V\setminus S}(u_1)=1<2= deg_{S}(u_1)  $, it is immediate to notice that
    	\begin{equation}
    		\label{e:u1 fixates}
    		\mP\bigl(\text{$ u_1 $ fixates at the value $ +1 $}\bigr)\geq\mP\bigl( \forall i=1,\dots,4 \ \sigma_0(u_{i})=+1 \bigr)=p^4>0.
    	\end{equation}
    	Hence, by ergodicity (see \cite{M1999, NNS2000}), there exist vertices that fixate at the value $ +1 $ almost surely.
    	Now, let $ z, w_i\in V $ be the vertices such that $ deg(z)=4 $  and $ deg(w_i)=3 $ for $ i=1,\dots,8 $ as in Figure \ref{f: exampleM}. Let $ E_z $ be the event that  the vertices $ u_i $ and $ w_i $ fixate respectively at the value $ +1 $ and $ -1 $ for each $ i=1,\dots,8 $.
    	With the same argument used in \eqref{e:u1 fixates}, we deduce that $ \mP(E_z)\geq p^{8}(1-p)^{8}>0 $. Moreover, conditioning on the event $ E_z $, whenever there is an arrival of the Poisson process $ \calP_z $ the spin flip at $ z $ occurs with probability  $ 1/2$. Thus, by L\'evy's extension of the Borel-Cantelli Lemmas, $ z $ flips infinitely often with positive probability. By ergodicity (see \cite{M1999, NNS2000}), 
 it follows that there exist vertices that flip infinitely often almost surely. Therefore, the  $ I(G,p) $-model is of type $ \calM $.             	
    \end{example}
   \section{Main results}
   \label{s: main results}

     In the following, given $ \sigma\in \Sigma $ and $ v\in V $, we denote by $ C_v(\sigma) $ the cluster at site $ v $ for $ \sigma $, defined as the maximal connected subset of $ V $ such that $ v \in C_v(\sigma) $ and for any $ u\in C_v(\sigma) $ one has $ \sigma(u)=\sigma(v) $.

   	Now, we present the following theorem, which is an extension of  Proposition 3.1 in \cite{CDN:clusters}.
   	\begin{theorem}
   		\label{pr:diverges}
   		For $ d\in \N $, take a $ I(G,p) $-model, where $ p\in[0,1] $ and $ G $ is a graph embedded in $ \R^{d} $  that is translation invariant with respect to $ d $ linearly independent vectors. Moreover, suppose that  $ G $ has the shrink property and $ \Delta(G)<\infty $. 
   		Then, the size of the cluster at a vertex $ v\in V $ diverges almost surely as $ t \to \infty $, i.e.
   		\begin{equation*}
   			\label{e: cluster diverges}
   			\forall v \in V, \quad \lim_{t \to \infty}|C_v(\sigma_t)|=\infty \quad \text{almost surely}.
   		\end{equation*}
   	\end{theorem}
   	\begin{proof}
 We  explicitly use the elements $\omega $ of  the sample space $\Omega$. 
We  prove the theorem by contradiction. Hence, for a vertex  $v \in V$, let us define the event  
$$
\mathcal{A} :  =\{ \omega \in \Omega :  \liminf_{t \to \infty}|C_v(\sigma_t)|<\infty   \}    .
$$
By contradiction assumption we  suppose $\mathbb{P} (\mathcal{A}) >0$. By  continuity 
of the measure there exists $M >0 $ such that 
$\mathbb{P} (\mathcal{A}_M) >0$, where 
$$
\mathcal{A}_M :  =\{ \omega \in \Omega :  \liminf_{t \to \infty}|C_v(\sigma_t)|<M \}. 
$$
 Then, for any 
$\omega \in \mathcal{A}_M$, one can define $ (T_k (\omega))_{k\in \mathbb{N}} $ such that 
\begin{equation*}
	\label{seqT}
T_1 (\omega)= \inf\{ t \geq 0:   |C_v(\sigma_t)|<M   \},
\end{equation*}
and, for $k \in \N$, one recursively defines
\begin{equation*}
	\label{seqTk}
T_{k+1} (\omega)= \inf\{ t \geq T_k(\omega) +1:   |C_v(\sigma_t)|<M \}.
\end{equation*}
     	Let $ \mathcal{F}_{t} $ be the $ \sigma $-algebra generated by the process up to time $ t $. It is immediate to note that  $ T_k $ is a stopping time with respect to the filtration $ (\mathcal{F}_t)_{t\geq 0} $ for any $ k\in\N $.
     	We consider $ \mathcal{F}_{T_{k}} $ for any $ k\in\N $. 
     	We notice that for $\omega \in \mathcal{A}_M$, since $ \Delta(G)<\infty $ and $ |C_v(\sigma_{T_k})|<M $, the cluster  $  C_v(\sigma_{T_k}) $ can be equal  only to a finite number of sets of vertices. 
     	For each of these sets of vertices, 
     	by the shrink property there is an ordered finite sequence of clock rings and outcomes of tie-breaking coin tosses inside a fixed finite ball that would cause the cluster to shrink to a single site $ w\in V $ with $ d_G(w,v)<M $ ($ w $ could, in principle, depend on $ C_v(\sigma_{T_k}) $). 
     	Then, since the vertex $ w $  would have all neighbours with opposite spins, it could  
     	 have an energy-lowering spin flip (with change in energy equal to $  -2 deg(w) $) and  the cluster would vanish with positive probability. 
     	 	We define
     	\begin{equation*}
     			\calB_{M,k}:=\bigcup_{\substack{w\in V:\\ d_G(w,v)<M} }\bigl\{\text{spin at $ w $ flips at time $ t\in (T_k,T_k+1) $ with $ \Delta \mathcal{H}_{w}(\sigma_t) \leq -1 $}\bigr\}.
     	\end{equation*}  	
     	From the previous statements in this proof, one has that there exists $ \delta >0 $ such that 
     	\begin{equation*}
     		\mP\bigl(\calB_{M,k}|\sigma_{T_k}\bigr)\geq \delta.
     	\end{equation*} 
     	Now, by the Strong Markov property of the process, for any $ k\in\N $ 
     	 we have  the following lower bound    	
     	\begin{equation*}
     		\label{e:xip2}
     		\xi_k (\omega) :=
     		 \mathbb{P}\bigl(\calB_{M,k}|\mathcal{F}_{T_k}\bigr) (\omega)
     		 =\mP\bigl(\calB_{M,k}|\sigma_{T_k}\bigr)\geq \delta,
     	\end{equation*}
for almost every $\omega \in \mathcal{A}_M$. 
        Thus $ \sum_{k=1}^{\infty}\xi_k (\omega )=\infty $, for almost every $\omega \in \mathcal{A}_M$. 
     	 Now, by using the L\'evy's extension of Borel-Cantelli Lemmas (see e.g. \cite{W1991}) with the sequence of events
     	$ (\calB_{M,k})_{k\in \N} $ and the filtration $ (\mathcal{F}_{T_{k}})_{k\in \N} $, we have that
     	\begin{equation*}
    \left  	\{	\omega \in \Omega :\sum_{k=1}^{\infty}\xi_k (\omega)=\infty  \right \} \subset  \left  	\{	\omega \in \Omega :\sum_{k=1}^\infty \mathbf{1}_{\calB_{M,k}}    (\omega)  =\infty \right \}\cup C 
     	\end{equation*}
        where $ C $ has measure zero.
         Then
     	\begin{equation*}
     		\mP\bigl(\limsup_{k \to \infty }\calB_{M,k}\bigr)\geq 
     		\mP\bigl(\mathcal{A}_M\bigr)
     		>0.
     	\end{equation*}
        
         Thus, there exists a vertex 
         $ \tilde{w} $ with $ d_G(\tilde{w},v)<M $ such that energy-lowering spin flips occur at $ \tilde{w} $ infinitely many times with positive probability. 
         By translation invariance with respect to $ d $ linearly independent vectors, $v_1, \ldots , v_d$, we obtain  that the graph $G$ is quasi-transitive. The classes of the graph are all represented inside the  parallelogram spanned by the $d$ vectors $v_1, \ldots , v_d$. Now,  the translation-ergodicity implies that  there exists a positive density of vertices for which energy-lowering spin flips occur infinitely often almost surely. This fact contradicts Theorem 3 and related remark in \cite{NNS2000} (see also  Lemma 5 in \cite{CDS18}).   
     	This concludes the proof.    	
   	\end{proof}
   	\begin{remark}
   		 Theorem \ref{pr:diverges} is a generalization of  Proposition 3.2 in \cite{CDN:clusters} for graphs $ G $  having the shrink property. In \cite{CDN:clusters}, the result was given only for the cubic lattice $ Z ^d $ that in particular has the shrink property. 
   	\end{remark}

       In the following, we consider the $ I(G,p) $-model having $G \in \G (a)$ for $ a\in \{3,4\} $ and it is invariant under translation with respect to $\bar x $. Without loss of generality we take $\bar x = (1,0 )$.
   Let us consider a  vertex $ \tilde{v}$ having minimal  Euclidean distance 
   from the origin $O$.  Clearly,  
   $  \tilde v =O$  when $O $ belongs to $ V $.   In the case $\tilde v \neq O $ we consider the two 
   distinct vertices $\tilde v ,  R_{ \theta }( \tilde v)$; 
   since  $ G $ is a connected graph, we can  select  a connected finite set 
   $S \subset V $  such that  
   $\tilde v ,   R_{ \theta } (\tilde v) \in  S$. Finally we define the set  of vertices 
    \begin{equation}
   		\label{e:U}
   		U =
   		\bigg \{
   		\begin{array}{ll}
   			\{O\} & \text{if } O \in V\\
   			\bigcup_{k=0}^{a-1} R_{ k\theta(a) } (S) &  \text{otherwise}, \\
   		\end{array}
   \end{equation}  
   where $\theta(a)  = 2 \pi / a $. 
   By construction $U $ is connected and $R_{ \theta(a) } (U)  = U $.
   
   \medskip

   For 
   $ a\in \{3,4\} $ we  construct a region of size $ L  \in \R_+$ centered in $ O $ as follows. 
   Let us consider the point $P_1(L , a) = (L \tan  (\theta(a) /2) ,L )\in \R^2 $ and let 
   $$
   P_{i +1}(L , a) = R_{ i\theta(a) } (P_1),
   $$
   for $i =1,  \ldots , a -1$. 
   We define the region of size $ L  \in \R_+$ centered in $ O $ as follows
   $$
   T_{L} (a ) := Conv (\{P_1(L , a) , \ldots , P_{a }(L , a) \}). 
   $$
   For $a=3,4$ one  respectively obtains that $T_{L} (a )$ is an equilateral triangle or a square.

Now, let us consider the class $V_1 \subseteq V $ (see Theorem \ref{t:quasi-transitive} and Definition \ref{d: quasi-transitive graph}). If the graph $G =(V, E) $ is transitive then   $V_1 = V$.  
 	We  write every vertex $ v\in V $ as $ v=(v_x,v_y )\in \R^2 $ and let $ v_{0,y}:=\max_{v\in  V_1 \cap T_{L  }(a) }v_y $ and $ v_{0,x}:=\min \{v_x \in \R:  (v_x , v_y)\in  V_1 \cap T_{L  }(a)  \text { and }    v_y =   v_{0,y}      \}$.
 We define $ v_0= ( v_{0,x},    v_{0,y} )\in V_1 \cap T_{L  }(a)    $.
By translation invariance of $G$ with  respect to $\bar x = (1,0)$, one has  $ v_0 +  \bar{x} \in V_1$. 
 Now we can select a connected set 
of vertices $U_0  $ such that $ V_1 \cap B (v_0, 2 ) \subset U_0$. Finally we choose $r_1 \geq 2 $ such that 
$   U_0 \subset  B (v_0, r_1) $.
    We are ready to present the following lemma.      
      \begin{lemma}
      	\label{l:cycle}
      	For any $ G\in \G(a) $ with $ a\in \{3,4\} $ and  for any $ L\in \R_+ $  there exists  a connected set of vertices
$W_{L}\subset (T_{L +2 r_1}(a) \setminus   T_{L-2 r_1 }(a) )$ such  that 
$ R_{ \theta(a) }(W_L) = W_{L}$.  	
      \end{lemma}
      \begin{proof}
 For $k \in \N$, let 
$$
v_k:= v_0 +k \bar{x} . 
$$
We define $k_{max} = \max \{k \in \N : v_k \in T_{L }(a)  \}$.  
The set of vertices 
$$
\calS = \bigcup_{k=0 }^{k_{max}} (U_0+k\bar{x})  
$$
is connected because for any $k =0, \ldots, k_{max } -1 $ it turns 
 out that $ G[U_0+k\bar{x}] $ is connected and $v_k, v_{k+1}\in U_0+k\bar{x} $.  

We define the set of vertices $ W_{L} =     \bigcup_{i=0 }^{a  -1}        R_{ i\theta(a) } (\calS)$. 
Now, we show that $ W_{L} $ is  connected. Since 
	\begin{equation*}
		 ||v_0-P_2(L,a)||\leq 1, \quad ||P_1(L,a)-v_{k_{max}}||\leq 1
	\end{equation*}
 and by using the triangular inequality, one obtains
	\begin{multline*}
		||v_0-R_{\theta(a)} (v_{k_{max}})||\leq ||v_0-P_2(L,a)||+||P_2(L,a)-R_{\theta(a)}(v_{k_{max}})||=\\
		=||v_0-P_2(L,a)||+||P_1(L,a)-v_{k_{max}}||\leq 2.
	\end{multline*}
	The previous inequality and $ V_1 \cap B (v_0, 2 ) \subset U_0$ imply that $ W_{L} $ is  connected.
Clearly $W_{L}\subset (T_{L +2 r_1}(a) \setminus   T_{L-2 r_1 }(a) )$ and $ R_{ \theta(a)}(W_{L}) = W_{L}$.   	
      	\end{proof}  
      Let  $ W_{L}$ as in Lemma \ref{l:cycle}. One can select a cycle 
   $ U_L \subset W_{L} $; we call $ f_{L,\infty} $ its outer face and $ f_{L,0} $ its inner face. 
   \begin{remark}
   	\label{r:face}
   	We notice that, by translational invariance with respect to the vectors $ \bar{x}=(1,0) $ and $ \bar{y}=(\cos\theta(a),\sin \theta(a)) $, 
   	one has
   	\begin{equation*}
   		\label{e: asymp}
   		|W_L|\asymp L \quad \text{and} \quad  |V\cap f_{L,0}|\asymp L^{2}, 
   	\end{equation*}
    where we write $ a_n\asymp b_n $ to mean that there exist two positive constants $ c_1 $ and $ c_2 $ such that $ c_1\leq \frac{a_n}{b_n}\leq c_2 $ for all $ n\in \N $.
   	  This implies that $ G\in \G(a) $ is amenable  for any $ a\in \{3,4,6\} $. Under the assumptions of amenability of the graph, the translation invariance of the measure $ \mu $  and  finite-energy of the measure $\mu$, it is known that  the infinite cluster is at most one almost surely (see \cite{BS2011, BK1989, HJ2006}). For the zero-temperature stochastic Ising model, it is not known whether the measure induced at time $ t $ has finite-energy property. Therefore, we are not able to prove the uniqueness of the infinite cluster at time $t$. Instead,  
if  the temperature is positive and decreases to zero, one has the property of finite-energy (see \cite{CDS18}). In this last case one obtains the uniqueness of the infinite cluster. 
   \end{remark}

     Now, given an integer $ q \in \N$ and $\delta < \frac{1}{2q} $, we consider $T_{1 +\delta }(a) \setminus   T_{1-\delta }(a)$ and we show that there exists a collection of balls $(B (c_i, 4/q): i = 1, \ldots , aq ) $ such 
that: 
\begin{itemize}
 \item[a.] $      T_{1 +\delta }(a) \setminus   T_{1-\delta }(a) \subset              \bigcup_{i =1}^{aq}B (c_i, 4/q) $; 
 \item[b.] for any $ i = 1, \ldots, aq $,  the center $c_i $ belongs to $\partial T_{1 }(a)$;
 \item[c.] for any $ i = 1, \ldots, q $ and $m =0, \ldots , a-1$ one has $c_{i+m q} = R_{ m\theta(a)}( c_i ) $. In particular,  $R_{ \theta(a)}(      \bigcup_{i =1}^{aq}B (c_i, 4/q)    ) = 
    \bigcup_{i =1}^{aq}B (c_i, 4/q) $. 
\end{itemize}
It is clear that such a construction exists, for example by taking the centers of the ball equally spaced. The chosen balls in this construction  will be  maintained also in the sequel. 
  
       \begin{figure}[htp]
  	\centering
  	\includegraphics[scale=0.40]{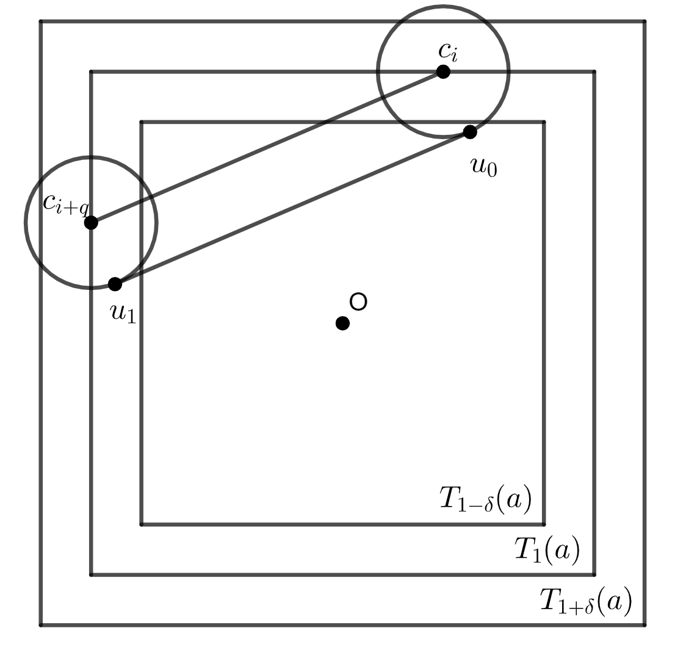}
  	\caption{  The distance between the segment having $u_0 \in B (c_i , 4/q)$ and  $u_1 \in B (c_{i+ q} , 4/q)$ as its endpoints and $ O $ can decrease at most of $4/q$ (the length of the radius)
  		with respect to the distance between the segment having endpoins $c_i$ and $  c_{i+ q}$ and $ O $. }
  	\label{f: glem}
  \end{figure}  

     \begin{lemma}[Geometric Lemma]
     	\label{l:convex hull}
Let $a\in\{3,4\}$, $q \geq 10$ and $\delta < \frac{1}{2q} $ and consider the  cover of $T_{1 +\delta }(a) \setminus   T_{1-\delta }(a)$ introduced in items a, b, and c. 
Then, for any $ (u_{k})_{k =0, \ldots, a-1} $ such that $ u_{k}   \in B (c_{i+kq}, 4/q) $ for $k =0, \ldots, a-1$,  one has 
  $ Conv(\{ u_0 , \ldots , u_{a-1}\})\supset B(O, \frac{1}{2} -\frac{4}{q}  )$. 
     \end{lemma}
     \begin{proof} 
       For a fixed $i = 1,\ldots , q $, let us consider    	 $ (c_{i +kq})_{k =0, \ldots , a-1}$.	For $ a=3,4 $, we note that $  Conv(\{ c_{i +kq}:k =0, \ldots , a-1\}) $ is an equilateral triangle or a square. Therefore, since $ c_{i +kq}\in \partial T_{1 }(a) $ one has $ Conv(\{ c_{i +kq}:k =0, \ldots , a-1\})\supset B(O, \frac{1}{2}  )$.
  Let us now consider the segment having $u_0 \in B (c_i , 4/q)$ and  $u_1 \in B (c_{i+ q} , 4/q)$ as its endpoints. The distance between this segment and the origin $O$ can decrease at most of $4/q$ 
with respect to the distance between $ O $ and the segment having endpoins $c_i ,  c_{i+ q}$ (see Figure \ref{f: glem}).  
Then one obtains that 
$ Conv(\{ u_0 , \ldots , u_{a-1}\})\supset B(O, \frac{1}{2} -\frac{4}{q}  )$. 
     \end{proof}
     As already announced, we do not deal with $ \theta=\pi $. Indeed if we consider $a=2$ which corresponds to $ \theta (a)=\pi $, 
     this statement is false because $ Conv(\{u_0,u_1\}) $ would be a segment and there is no ball contained in it. From now on we take $q\geq 24 $ and hence   $\frac{1}{2} -\frac{4}{q}  \geq  \frac{1}{3}$.
      
          \begin{figure}[htp]
      	\centering
      	\includegraphics[scale=0.40]{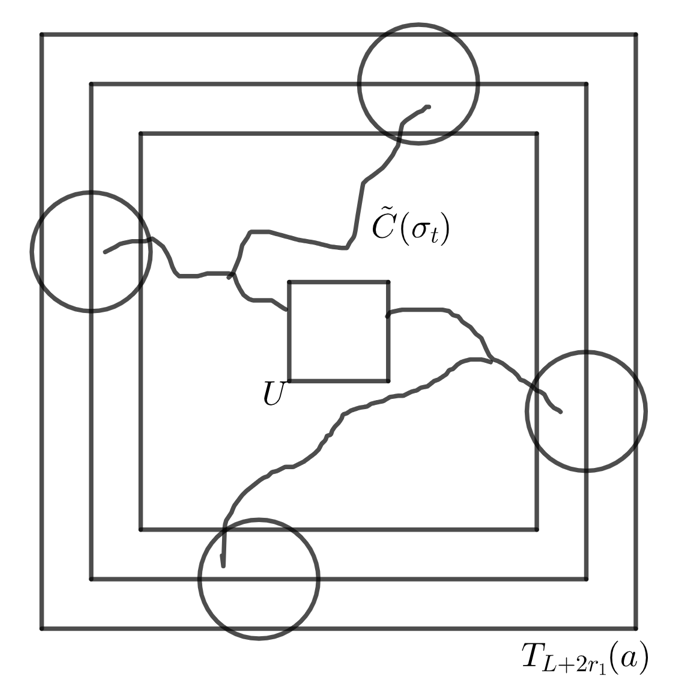}
      	\caption{  Example of a realization of an $ L $-cross.
      	 }
      	\label{f: lcross}
      \end{figure}  
  
      Now, we present the following definition.
     \begin{definition}[$ L $-Cross] 
     	\label{d:cross}
     	Given $ G\in \G(a) $ with $ a\in \{3,4\} $, we say that an 
     	 \emph{$L$-cross of $ +1 $} occurs at time $ t $  if  there exists a cluster $ \tilde{C}(\sigma_t) $ of $ G[V \cap T_{L+ 2 r_1}(a)] $ such that 
     	\begin{itemize}
     		\item $ \sigma_t(v)=+1 $ for each $ v\in \tilde{C}(\sigma_t) $;
     		\item $ \tilde{C}(\sigma_t)\supset U $, where $ U $ is defined in \eqref{e:U};
     		\item there exists $i   \in \{1,\dots,q\} $ such that $ \tilde{C}(\sigma_t)\cap B (L c_{i+ k q}      ,  \,  4 L /q    )\neq \emptyset $ for each $ k=0,\dots,a -1 $.
     	\end{itemize}
         We denote by $ E^{t}_{L} $ with $ t\in \R_{0}^{+} $ the event that an 
          $L$-cross of $ +1 $ occurs at time $ t $. Moreover, we define 
\begin{equation}\label{AL}
A_L:=\limsup_{t\to \infty}E^t_{L} . 
\end{equation}
     \end{definition}

     We define the set of vertices 
     $ S_{L}(t):=\tilde{C}(\sigma_t)\cap W_{L} $, where the  properties of $ W_{L} $ are given in Lemma \ref{l:cycle}. 
      The previous Lemma \ref{l:convex hull} shows that for each time $ t\in \R_{0}^{+} $ in which an 
      $L$-cross of $ +1 $ occurs  (see Figure \ref{f: lcross}), one has
     \begin{equation}\label{rag}
     	Conv(S_{L}(t))\supset  B\left (O, r_L \right ) \quad \text{where $ r_L=\frac{1}{3}L $}.
     \end{equation}
     In other words, Lemma \ref{l:convex hull} says that $ |Conv_{G}(S_{L}(t))|\asymp L^2 $. 
     

     \begin{lemma}
     	\label{l:io}
     	Consider the $ I(G,p) $-model, where $ p\in[\frac{1}{2},1) $ and $ G\in \G$.
     	If $ G $ has the shrink property, then 
$$
\mP(A_L)\geq p_{cross}:= \frac{1}{(aq)^{a}}\left (\frac{1}{2}\right )^{a|U|} , 
$$
where the event $A_L  $ is defined in \eqref{AL}. 
     \end{lemma}

  We now explain the strategy for proving Lemma \ref{l:io}. Let $ U\subset V $ as defined in \eqref{e:U}. If the initial density $ p\geq 1/2 $, then   $ U  $ is  contained in a cluster of $ +1 $
     	with  probability at least $ (1/2)^{|U|} $, at any time $ t $.
     	By  Theorem \ref{pr:diverges}, the size of this cluster  diverges almost surely as $ t \to \infty $. Now, by FKG inequality and by rotation invariance, one obtains that $ \liminf_{t \to \infty}\mP(E^{t}_{L})>0 $, i.e. the cluster satisfies the properties in Definition \ref{d:cross} with positive probability. Note that this lower bound does not depend on $ L $. By Reverse Fatou Lemma, we obtain the same lower bound on $ \mP(A_L) $.  We are now ready to prove the lemma.

     \begin{proof}
     	Let $ \calU_{t} $ be the event that all vertices in $ U $ have spin equal to $ +1 $ at time $ t $.      	
     	By Lemma \ref{l:order coupling}, FKG inequality and Harris' inequality (see \cite{L1985, H1977}), it follows that 
     	\begin{equation}
     		\label{e: vertices in U fixate}
     		\mP(\calU_{t})\geq \biggl(\frac{1}{2}\biggr)^{|U|}.
     	\end{equation}
     	If  $ \calU_{t} $ occurs, then, since $ U $ is connected,
     	 at time $ t $ all vertices in $ U $ belong to a same cluster, 
     	we call it  $ C_{U}(\sigma_t) $. Moreover, let $ \tilde{C}(\sigma_t)$ be the cluster of $  G[V \cap T_{L+ 2 r_1}(a)] $ that contains $ U $. 
     	We denote by $ \calC_{W_L}(t) $ the event that the cluster $ \tilde{C}(\sigma_t) $ intersects $  W_L $, i.e.  $ \calC_{W_L}(t):=\{\tilde{C}(\sigma_t)\cap W_L\neq \emptyset\} $.
     	By Theorem \ref{pr:diverges}, we have that $ \lim_{t \to \infty}|C_{U}(\sigma_t)|=\infty $ almost surely. 
     	Thus, by planarity of $ G $ and $ V\cap f_{L,0} $ has finite cardinality (see Remark \ref{r:face}), 
     	we get
     	\begin{equation}
     		\label{e: cons. prop.1}
     		\lim_{t \to \infty}\mP\bigl(\calC_{W_L}(t)\bigr)=1.
     	\end{equation}
     	By \eqref{e: vertices in U fixate} and \eqref{e: cons. prop.1}, it follows that
     	\begin{equation}
     		\label{e: lower bound liminf with conditional probability}
     		\liminf_{t \to \infty}\mathbb{P}\bigl(\calC_{W_L}(t)\cap \calU_{t} \bigr)=\liminf_{t \to \infty}\mP(\calU_{t})\geq \biggl(\frac{1}{2}\biggr)^{|U|}.
     	\end{equation}
       Now, we write $ W_L=\cup_{i=1}^{q}\cup_{k=0}^{a-1}P_{L,k}^{i}$ where $ P_{L,k}^{i}:= W_L\cap B (L c_{i+ k q},  \,  4 L /q    )$ for  $ i=1,\dots,q$ and $ k=0,\dots,a-1 $. 
        	We define the event $ \calC_{L,i,k}(t):=\{\tilde{C}(\sigma_t)\cap P_{L,k}^{i}\neq \emptyset\} $ for   $ i=1,\dots,q$ and $ k=0,\dots,a-1 $.
        	Hence, we have 
        	\begin{equation*}
        		\calC_{W_L}(t)\cap \calU_{t}=\bigcup_{i=1}^{q}\bigcup_{k=0}^{a-1}\calC_{L,i,k}(t)\cap \calU_{t}. 
        	\end{equation*}
        	Thus, by rotation invariance 
        	and by the union bound, we have
        	\begin{multline}
        		\label{eq:union bound 3}
        		\mathbb{P}\bigl(\calC_{W_L}(t)\cap \calU_{t} \bigr)=\mP\biggl(\bigcup_{i=1}^{q}\bigcup_{k=0}^{a-1}\bigl(\calC_{L,i,k}(t)\cap \calU_{t} \bigr)\biggr)\leq \sum_{i=1}^{q}\sum_{k=0}^{a-1}\mP\bigl(\calC_{L,i,k}(t)\cap \calU_{t} \bigr)\leq \\
        		\leq aq\mP\bigl(\calC_{L,\bar{i},0}(t)\cap \calU_{t} \bigr), 
        	\end{multline}
        	where $ \bar{i}\in \{1,\dots,q\} $ is such that $\mP\bigl(\calC_{L,\bar{i},0}(t)\cap \calU_{t} \bigr)=\max_{i=1,\dots,q}\mP\bigl(\calC_{L,i,0}(t)\cap \calU_{t} \bigr)$. 
        	We note that 
        	$ \calC_{L,\bar{i},k}(t)\cap \calU_{t} $ is an increasing event for  $ k=0,\dots,a-1 $; 
        	therefore 
        	\begin{equation}
        		\label{e: fkg lower bound}     			\mP\biggl(\bigcap_{k=0}^{a-1}\bigl(\calC_{L,\bar{i},k}(t)\cap \calU_{t} \bigr)\biggr)\geq \biggl(\mP\bigl(\calC_{L,\bar{i},0}(t)\cap \calU_{t} \bigr)\biggr)^{a}\geq \biggl(\frac{1}{aq}\mathbb{P}\bigl(\calC_{W_L}(t)\cap \calU_{t} \bigr)\biggr)^{a},
        	\end{equation}
            where the first inequality follows by FKG inequality and  
            by rotation invariance, and the last inequality follows by  \eqref{eq:union bound 3}.
        	We also notice that, by definition of $ P_{L,k}^{i} $, one has
        	\begin{equation*}
        		\calC_{L,i,k}(t)=\{\tilde{C}(\sigma_t)\cap P_{L,k}^{i}\neq \emptyset\}\subset \{\tilde{C}(\sigma_t)\cap B (L c_{i+ k q},  \,  4 L /q    )\neq \emptyset\}.
        	\end{equation*}
        	Thus, by Definition \ref{d:cross}, we have
        	\begin{multline*}
        		E^t_{L}=\bigcup_{i=1}^{q}\bigcap_{k=0}^{a-1}\bigl\{\tilde{C}(\sigma_t)\cap B (L c_{i+ k q},  \,  4 L /q    )\neq \emptyset\bigr\}\cap \calU_{t} \supseteq \\
        		\supseteq\bigcup_{i=1}^{q}\bigcap_{k=0}^{a-1}\calC_{L,i,k}(t)\cap \calU_{t}\supseteq\bigcap_{k=0}^{a-1}\calC_{L,\bar{i},k}(t)\cap \calU_{t},	
        	\end{multline*}
        	and hence
        	\begin{equation}
        		\label{e:plb}
        		\mathbb{P}(E^t_{L})\geq \mP\biggl(\bigcap_{k=0}^{a-1}\calC_{L,\bar{i},k}(t)\cap \calU_{t} \biggr).
        \end{equation}    	 	
     	Now, by \eqref{e: lower bound liminf with conditional probability}, \eqref{e: fkg lower bound} and \eqref{e:plb}, we obtain the following lower bound
     	\begin{equation*}
     		\label{eq:cross 2}
     		\liminf_{t \to \infty}\mathbb{P}(E^t_{L})\geq\liminf_{t \to \infty}\biggl(\frac{1}{aq}\mathbb{P}\bigl(\calC_{W_L}(t)\cap \calU_{t} \bigr)\biggr)^{a}\geq \frac{1}{(aq)^{a}}\left (\frac{1}{2}\right )^{a|U|}.
     	\end{equation*}
     	Finally, by Reverse Fatou Lemma we get 
     	\begin{equation*}
     		\label{e: limsup E lower bound}
     		\mathbb{P}(A_L)\geq \limsup_{t\to \infty}\mathbb{P}(E^t_{L})\geq \liminf_{t \to \infty}\mathbb{P}(E^t_{L})\geq\frac{1}{(aq)^{a}}\left (\frac{1}{2}\right )^{a|U|}>0  
     	\end{equation*}     
     	that concludes the proof.
     \end{proof}

Now  we give a simple definition that will be useful when related to $ E^t_L$ through Lemma \ref{l:convex hull}.
 For  $ t\in \R_{0}^{+} $, let $ F^{t}_{L} $ be   the event that all sites belonging to $ B(O, \frac{L}{3})$  have spin equal to $ +1 $ at  some time $ s \in(t, t+1)$. 
     \begin{lemma}
     	\label{l:convex hull and cross with +1}
     	Consider the $ I(G,p) $-model, where $ p\in[\frac{1}{2},1) $ and $ G\in \G$. If $ G $ has the planar shrink property, then there exists $ \epsilon_{L}>0 $ such that  
     	\begin{equation*}
     		\label{e: lower bound conditional probability}
     		\mathbb{P}(F^{t}_{L}|\sigma_t=\sigma)\geq \epsilon_{L} 
     	\end{equation*}
     	for any $ \sigma\in \Sigma $ such that $ \{\sigma_t=\sigma\}\subset  E^{t}_{L}$. 
     \end{lemma}
     \begin{proof} 	
     	Let $ \sigma\in \Sigma $ 
     	and $ (\sigma_s)_{s\geq 0} $ be the  $ I(G,p) $-model such that $ \{\sigma_t=\sigma\}\subset  E^{t}_{L} $. We  define another zero-temperature stochastic Ising model $ (\sigma_s^{\prime})_{s\geq t} $ with infinitesimal generator having the flip rates as in \eqref{e:our rates} and such that
     	\begin{equation*}
     		\sigma_t^{\prime}(v)= \begin{cases}
     			+1  & \text{for each  $ v\in \tilde{C}(\sigma_t)$} \\
     			-1 & \text{otherwise},
     		\end{cases}
     	\end{equation*}   
     	where $ \tilde{C}(\sigma_t) $ is the cluster of $ G[V \cap T_{L+ 2 r_1}(a)] $ in the configuration $ \sigma $, as in Definition \ref{d:cross}.
     	By definition of $ \sigma_t^{\prime} $ and $ E^{t}_{L} $, $ \tilde{C}(\sigma_t)\supset U $ is the unique cluster of $ +1 $ sites in the configuration $ \sigma_t^{\prime}$.  
     	In configuration $  \sigma_t^{\prime} $, we have that $ \bigl(V \cap T_{L+ 2 r_1}(a)\bigr)\setminus \tilde{C}(\sigma_t)=D_1(\sigma_t^{\prime})\sqcup \dots \sqcup D_{k}(\sigma_t^{\prime}) $, where $ D_i(\sigma_t^{\prime})$ for $ i=1,\dots,k  $ are clusters of $ -1 $ sites (we stress that $ k<\infty $ because $ V \cap T_{L+ 2 r_1}(a)  $ has finite cardinality).

     	We notice that, by planarity of $ G $, for each $ i=1,\dots,k $ we have that $ \partial_{ext}D_i(\sigma_t^{\prime})\subset \tilde{C}(\sigma_t) $.
     	By planar shrink property, for each $ i=1,\dots,k $ there exists an ordered finite sequence of updates (i.e. of clock rings and outcomes of tie-breaking coin tosses inside $ V \cap T_{L+ 2 r_1}(a)$) 
     	that would cause all sites of 
$ D_i(\sigma_t^{\prime})\cap Conv_{G}(S_{L}(t))  $ (see definition above formula \eqref{rag}) to 
  have spin equal to $ +1 $ in some time $s \in (t, t+1)$ with positive probability. Therefore, we get an ordered finite sequence of updates  inside $ V \cap T_{L+ 2 r_1}(a)$  that would cause all sites of $  Conv_{G}(S_{L}(t))  $ to have spin equal to $ +1 $ in   $  \sigma_s^{\prime} $ (with $s \in (t, t+1)$), but since $ \sigma_t^{\prime}\leq \sigma_t $  this  sequence of updates works, by the coupling in Lemma \ref{l:order coupling}, in the same way for  the original process  
$( \sigma_s : s \in (t, t+1))$.  
Moreover, by Lemma 	\ref{l:convex hull}, one has $V\cap B(O, \frac{L}{3})     \subset Conv_{G}(S_{L}(t)) $.

     	Thus,  there exists $ \epsilon_{L}>0 $ such that
     	\begin{equation*}
     		\mathbb{P}(F^{t}_{L}|\sigma_t=\sigma)\geq \epsilon_{L}
     	\end{equation*}
     	for any $ \sigma\in \Sigma $  having $ \{\sigma_t=\sigma\}\subset  E^{t}_{L}$.  
     \end{proof}  
     We recall that $ A_L:=\limsup_{t\to \infty}E^t_{L}  $. Now, we define $ B_L:=\limsup_{t\to \infty}F^t_{L} $. We are ready to present the following lemma.

     \begin{lemma}
     	\label{p:lower bound with cross}
     	Consider the $ I(G,p) $-model, where $ p\in[\frac{1}{2},1) $ and $ G\in \G$. If $ G $ has the planar shrink property, then 
     	$ \mP(B_L)\geq p_{cross}>0 $. 
     \end{lemma}
     \begin{proof} In the proof we will explicitly use the elements $\omega $ of  the sample space $\Omega$.  
     	First, let  $ \omega \in A_L $, i.e. an $L$-cross of $ +1 $ occurs  infinitely often. Then one can define $ (T_k (\omega))_{k\in \mathbb{N}} $ such that 
\begin{equation*}\label{seqT}
T_1 (\omega)= \inf\{ t \geq 0:    E^{t}_L  \text{ occurs} \},
\end{equation*}
and, for $k \in \N$, we recursively define
\begin{equation*}\label{seqTk}
T_{k+1} (\omega)= \inf\{ t \geq T_k +1:    E^{t}_L  \text{ occurs} \}.
\end{equation*}
     	Let $ \mathcal{F}_t $ be the $ \sigma $-algebra generated by the process up to time $ t $. It is immediate to note that  $ T_k $ is a stopping time with respect to the filtration $ (\mathcal{F}_t)_{t\geq 0} $ for any $ k\in\N $. 
     	We consider $ \mathcal{F}_{T_{k}} $ for any $ k\in\N $.
     	By the Strong Markov property of the process and by Lemma \ref{l:convex hull and cross with +1}, for any $ k\in\N $ we have  the following lower bound
     	\begin{equation*}\label{e21}
     		\xi_k (\omega) :=
     		 \mathbb{P}\bigl(F^{T_k}_{L}|\mathcal{F}_{T_k}\bigr) (\omega) =\mP\bigl(F^{T_k}_{L}|\sigma_{T_k}\bigr)\geq \epsilon_{L}>0,
     	\end{equation*}
for almost every $\omega \in A_L$. 
        Thus $ \sum_{k=1}^{\infty}\xi_k (\omega )=\infty $, for almost every $\omega \in A_L$. 
     	 Now, by using the L\'evy's extension of Borel-Cantelli Lemmas (see e.g. \cite{W1991}) with the sequence of events
     	$ (F^{T_k}_{L})_{k\in \N} $ and the filtration $ (\mathcal{F}_{T_{k}})_{k\in \N} $, we have that
     	\begin{equation*}
    \left  	\{	\omega \in \Omega :\sum_{k=1}^{\infty}\xi_k (\omega)=\infty  \right \} \subset  \left  	\{	\omega \in \Omega :\sum_{k=1}^\infty \mathbf{1}_{F^{T_k}_{L}}    (\omega)  =\infty \right \}\cup C, 
     	\end{equation*}
        where $ C $ has measure zero. 
     	Then, by 
     	 Lemma \ref{l:io}, 
     	 we get
     	\begin{equation*}
     		\label{e: lower bound F}
     			\mP\bigl(B_{L}\bigr)\geq \mP\bigl(A_L\bigr) \geq \frac{1}{(aq)^{a}}\left (\frac{1}{2}\right )^{a|U|}= p_{cross}>0. 
     	\end{equation*}
     	This concludes the proof.   
     \end{proof}
         
         Now, for any time $ t_{2}>t_{1} +1$  we define
     	\begin{equation*}
     	D(L;t_{1},t_{2}):=\bigcup_{s\in[t_{1} ,t_{2} -1 ]}F^{s}_{L}\quad \text{and} \quad  D(L;t_{1},\infty):=\bigcup_{s\geq t_{1}}F^{s}_{L}.
     \end{equation*}
     \begin{lemma}
     	\label{l:event D}
     	For any $ L\in \R_+ $ and $ t_1\geq 0 $, one has 
     	\begin{equation*}
     		\mP\bigl(D(L;t_{1},\infty)\bigr)\geq p_{cross}.
     	\end{equation*}
     	Moreover, for any $ \epsilon>0 $ there exists a time $ s>t_{1}+1 $  such that
     	\begin{equation*}
     		\mP\bigl(D(L;t_{1},s)\bigr)\geq (1- \epsilon) p_{cross}.
     	\end{equation*}
     \end{lemma}
     \begin{proof}
     	 We observe that $ D(L;t_{1},\infty)\supset B_{L} $. In particular, by Lemma \ref{p:lower bound with cross}, 
     	\begin{equation*}
     		\mP\bigl(D(L;t_{1},\infty)\bigr)\geq p_{cross}.
     	\end{equation*}
     	Now, we note that for $ t_{2}\leq t_{2}^{\prime} $ one has $ D(L;t_{1},t_{2})\subset D(L;t_{1},t_{2}^{\prime})$. 
     	 Thus, by continuity of  measure
     	 \begin{equation*}
     	 \lim_{t_2\to \infty}\mP\bigl(D(L;t_{1},t_2)\bigr)=	\mP\biggl(\bigcup_{s\geq t} D(L;t_{1},s)\biggr)=\mP\bigl(D(L;t_{1},\infty)\bigr)\geq p_{cross}.
     	 \end{equation*}
     	 Hence for all $ \epsilon>0 $ there exists a time $ s>t_{1}+1 $  such that
     	\begin{equation*}
     		\label{e: probability D lower bound 2}
     		\mP\bigl(D(L;t_{1},s)\bigr)\geq (1- \epsilon) p_{cross}.
     	\end{equation*}     	
     	\end{proof}
     Let $ \F $ be the  $ \sigma $-algebra generated by the process $ (\sigma_t)_{t\geq 0} $.
       All the events introduced  belong to $ \F $. Given a non-zero vector $ \bar{v} $ such that $G$ is  translation invariant  with respect to $ \bar{v} $, we define the configuration translated with respect to $ \bar{v} $ as
       \begin{equation*}
       	(\sigma+\bar{v})(v):=\sigma(v+\bar{v}) \quad \text{for any $ v\in V $}.
       \end{equation*}
       Let $ X $ be a $ \F $-measurable random variable. Then $ X=f((\sigma_t)_{t\geq 0}) $ 
       where $ f $ is a measurable function. We define 
       \begin{equation}
       	\label{e:operator v}
       	X+\bar{v}:= f((\sigma_t-\bar{v})_{t\geq 0}). 
       \end{equation}
       If $ X $ is an indicator function then $ f $ takes only the values $ 0$ or $1$. Let $ A\in \F $, one can define 
       \begin{equation*}
       	\mathbf{1}_{A}+\bar{v}= f((\sigma_t-\bar{v})_{t\geq 0})=: \mathbf{1}_{A+\bar{v}},  
       \end{equation*}
       that defines $ A+\bar{v} $ for any $ A\in \F $.
       
In the following  result we will  apply the ergodic theorem. We note that these processes are ergodic with respect to the translation  if the initial conditions are given for instance by a Bernoulli product measure, 
 see e.g. \cite{H1974,L1985, M1999} and references therein. 
 Now, we introduce some notation, which we will use in the proof of Theorem \ref{t:main result}. For $ v\in V $ and $ t\in \R_{+} $, let
 	\begin{equation*}
 	A^{+}_{v}(t):=\{\sigma_s(v)=+1, \ \ \forall s\in [0,t]\}, \quad 	A^{-}_{v}(t):=\{\sigma_s(v)=-1, \ \ \forall s\in [0,t]\}.
 \end{equation*}
    	 We denote by $ A^{+}_{v}(\infty) $ (resp. $A^{-}_{v}(\infty) $) the event that the vertex $ v $ fixates at the value $ +1 $ (resp. $ -1 $) from time zero. Clearly, $A_{v}^\pm(t) \subset A_{v}^\pm(t')$, for any $t' \leq t $. 	We recall that $ \{V_1, \dots, V_N \} $ is the partition of the vertex set $ V $ that comes from the quasi-transitivity of $ G\in \G $, see Theorem \ref{t:quasi-transitive}. 		We note that, since $ G $ is quasi-transitive, $ \mP( A^{\pm}_{v}(t)) $ depends only on the class to which the vertex $ v $ belongs and does not depend explicitly on the vertex itself. Thus,  for $ p=1/2 $, for each $ i=1,\dots, N $,  $ v\in V_i $, and $ t\in \R_{+}\cup \{\infty\} $,  we  set
    	 \begin{equation*}
    	 	\rho_{i}(t):=\mP( A^{+}_{v}(t))=\mP( A^{-}_{v}(t)).
    	 \end{equation*}
        The last equality follows by symmetry under the global spin flip for $ p=1/2 $.
     Now, we are ready to prove our main result. 
     \begin{theorem}
     	\label{t:main result}
     	If $ G=(V,E) \in \G$ has the planar shrink property, then the $ I(G,1/2) $-model is of type $ \I $, i.e., 
     	 all sites flip infinitely often almost surely.
     \end{theorem}
     \begin{proof}
     	We will prove the  theorem by contradiction. Suppose that there exists $ j\in \{1,\dots, N\} $ such that $ \rho_{j}(\infty)>0 $ that by Lemma \ref{l:it fixates from time 0} is equivalent to have a site that fixates with positive probability. We choose the following constants: $ \epsilon=\frac{1}{3}p_{cross}\rho_{j}(\infty) $, $ \epsilon_{1}=\frac{\rho_{j}(\infty)}{5} $, $\epsilon_{2}=\frac{1}{4}p_{cross} $ and $ \tilde{\epsilon}=\frac{1}{8} $. 

     	We notice that, by continuity of the  measure, the limit of $ \rho_{i}(t) $ as $ t \to \infty  $ exists and is equal to
     	\begin{equation*}
     		\lim_{t \to \infty}\rho_{i}(t)=\lim_{t \to \infty} \mP( A^{+}_{v}(t))=\mP\biggl(\bigcap_{m=1}^{\infty}A^{+}_{v}(m)\biggr)=\mP( A^{+}_{v}(\infty))=\rho_{i}(\infty),
     	\end{equation*}
     	for each $ v\in V_i$.
     	This implies that  there exists a time $ t_{\epsilon}>0 $ such that 
     	\begin{equation}
     		\label{e: t epsilon}
     		0 \leq \rho_{j}(t_{\epsilon})-\rho_{j}(\infty)<\epsilon.
     	\end{equation}
     Since $ G\in \G $, there exist  two linearly independent vectors $ \bar{x}_1 $ and $ \bar{x}_2 $ such that $ G $ is translation invariant with respect to them.
     	We want to construct on the graph $ G $ disjoint regions of a suitable size $ L_{0} $ centered 
     	in $ n_1\bar{x}_1+n_2\bar{x}_2 $ with $ n_1,n_2\in \Z $. 
     	By ergodicity (see \cite{H1974,M1999, NNS2000}), one has 
     	\begin{equation}
     		\label{e:ergodic result}
     		\lim_{r \to \infty}\frac{1}{n(r,j)}\sum_{v\in B(O,r)\cap V_j }\mathbf{1}_{A^{-}_{v}(t_{\epsilon})}=\rho_{j}(t_{\epsilon})>0 \quad \text{almost surely,}
     	\end{equation}
     	where $ n(r,j):=|B(O,r)\cap V_j| $.
     	Thus, \eqref{e:ergodic result} implies that  
     	there exists $ \tilde{r}\in \R_{+} $ 
     	such that  
     	\begin{equation*}
     		\mP\biggl(\frac{1}{n(\tilde{r},j)}\sum_{v\in B(O,\tilde{r})\cap V_j }\mathbf{1}_{A^{-}_{v}(t_{\epsilon})}\notin [\rho_{j}(t_{\epsilon})-\epsilon_{1},\rho_{j}(t_{\epsilon})+\epsilon_{1}]  \biggr)\leq \epsilon_{2}.
     	\end{equation*}
     	Then, in particular
     	\begin{equation}
     		\label{e: convergence in probability}
     		\mP\biggl(\sum_{v\in B(O,\tilde{r})\cap V_j }\mathbf{1}_{A^{-}_{v}(t_\epsilon)}
<n(\tilde{r},j)\bigl(\rho_{j}(t_\epsilon)-\epsilon_{1}\bigr)  \biggr)\leq \epsilon_{2}.
     	\end{equation}
     	Now, we construct  disjoint regions of size $ L_{0} $ on the graph $ G $ in the following way. 
     	Let  
     	$L_0 = 3 \tilde{r} $,  where $ L_{0} $ and $ \tilde{r} $ play the same role of $ L $ and $ r_L $ in 
     	\eqref{rag}.         
     	We define the  event
     	\begin{equation*}
     		G(L;t,\eta ):=\biggl\{\sum_{v\in B(O,L/3)\cap V_j }\mathbf{1}_{A^{-}_{v}(t)}\geq n(L/3,j)\bigl(\rho_{j}(t)-\eta \bigr)  \biggr\},
     	\end{equation*}
where $L, t, \eta >0$. 
      	By \eqref{e: convergence in probability}, one has 
     	\begin{equation}
     		\label{e: probability bad boxes}
     		\mP\bigl(G(L_0;t_\epsilon,\epsilon_{1})\bigr)\geq 1-\epsilon_{2}.    		 
     	\end{equation}

     	Now, let 
     	\begin{equation*} 
     		Y_{L_{0}}(t):=\sum_{v\in B(O,L_{0}/3)\cap V_j}\mathbf{1}_{A^{-}_{v}(t)}.
     	\end{equation*}
     Let $ n_0\in \N $ such that $ T_{L_{0}+ 2 r_1}(a)\cap (T_{L_{0}+ 2 r_1}(a)+n_0\bar{x}_i)=\emptyset $ for $i =1,2$. We define  $ Y_{L_{0},m_1,m_2}(t):= Y_{L_{0}}(t)+m_1n_0\bar{x}_{1}+m_2n_0\bar{x}_2 $ with $ m_1,m_2\in \Z $ (see \eqref{e:operator v}).
     	By ergodicity, it follows that for any $ t\in \R_{+} $
     	\begin{equation}
     		\label{ergodic result Y}
     		\lim_{M \to \infty}\frac{1}{(2M+1)^{2}}\sum_{\substack{m_1,m_2\in \Z:\\|m_1|,|m_2|\leq M}}\frac{1}{n(L_{0}/3,j)}Y_{L_{0},m_1,m_2}(t)=\rho_{j}(t), \quad \text{a.s.} 
     	\end{equation}

       We define the translated events
       \begin{equation*}
       	D(L;m_1,m_2;t_{1},t_{2}):=D(L;t_{1},t_{2})+m_1n_0\bar{x}_{1}+m_2n_0\bar{x}_2	
       \end{equation*}
      and
      \begin{equation*}
      G(L;m_1,m_2;t,\eta )	:=G(L;t,\eta )+m_1n_0\bar{x}_{1}+m_2n_0\bar{x}_2.
      \end{equation*}

            Now, by Lemma \ref{l:event D} and by translation invariance, 
            there exists a time $ {t}_{\tilde \epsilon}>t_{\epsilon}+1 $  such that
       \begin{equation}
       \label{e: probability D lower bound}
       	 \mP\bigl(D(L_{0};m_1,m_2;t_{\epsilon},t_{\tilde \epsilon})\bigr)\geq (1- \tilde{\epsilon}) p_{cross}.
       \end{equation}
 
     	By ergodicity, \eqref{e: probability bad boxes} and \eqref{e: probability D lower bound}, it follows that
     	\begin{multline}
     		\label{m: ergodic D int G}
     		\lim_{M \to \infty}\frac{1}{(2M+1)^{2}}\sum_{\substack{m_1,m_2\in \Z:\\|m_1|,|m_2|\leq M}}\mathbf{1}_{G(L_0;m_1,m_2;t_\epsilon,\epsilon_{1})\cap D(L_{0};m_1,m_2;t_{\epsilon},t_{ \tilde \epsilon})}=\\ =\mP\bigl(G(L_0;m_1,m_2;t_\epsilon,\epsilon_{1})\cap D(L_{0};m_1,m_2;t_{\epsilon},t_{\tilde \epsilon})\bigr)\geq \\ 
     		\geq \mP\bigl(D(L_{0};m_1,m_2;t_{\epsilon},t_{\tilde \epsilon})\bigr)-\mP\bigl(G(L_0;m_1,m_2;t_\epsilon,\epsilon_{1})^{c}\bigr)\geq (1-\tilde{\epsilon})p_{cross}-\epsilon_{2} \quad \text{a.s.}
     	\end{multline}

     Over the event 
     $   G(L_0;m_1,m_2;t_\epsilon,\epsilon_{1}) $ one has 
     	\begin{equation} \label{dopo}
     		Y_{L_{0},m_1,m_2}(t_{\epsilon}) \geq n(L_{0}/3,j)\bigl(\rho_{j}(t_\epsilon)-\epsilon_{1}\bigr). 
     	\end{equation}

By \eqref{ergodic result Y}, we get 
     	\begin{multline}
     		\label{e: rho difference}
     		\rho_{j}(t_{\epsilon})-\rho_{j}(  t_{ \tilde \epsilon}   )=\lim_{M \to \infty}\frac{1}{(2M+1)^{2}}\sum_{\substack{m_1,m_2\in \Z:\\|m_1|,|m_2|\leq M}}\frac{1}{n(L_{0}/3,j)}\biggl[Y_{L_{0},m_1,m_2}(t_{\epsilon})-Y_{L_{0},m_1,m_2}(t_{ \tilde \epsilon} )\biggr]\geq\\
     		\geq \lim_{M \to \infty}\frac{1}{(2M+1)^{2}}\sum_{\substack{m_1,m_2\in \Z:\\|m_1|,|m_2|\leq M}}\frac{1}{n(L_{0}/3,j)}Y_{L_{0},m_1,m_2}(t_{\epsilon})\mathbf{1}_{G(L_0;m_1,m_2;t_\epsilon,\epsilon_{1})\cap D(L_{0};m_1,m_2;t_{\epsilon},    t_{ \tilde \epsilon}     )}.
     	\end{multline}
     The inequality in \eqref{e: rho difference} follows by
     \begin{equation*}
     	\begin{split}
     		  	Y_{L_{0},m_1,m_2}(t_{\epsilon})-Y_{L_{0},m_1,m_2}(t_{ \tilde \epsilon} )&\geq Y_{L_{0},m_1,m_2}(t_{\epsilon})\mathbf{1}_{ D(L_{0};m_1,m_2;t_{\epsilon},    t_{ \tilde \epsilon}     )}\geq \\
     		  	& \geq Y_{L_{0},m_1,m_2}(t_{\epsilon})\mathbf{1}_{G(L_0;m_1,m_2;t_\epsilon,\epsilon_{1})\cap D(L_{0};m_1,m_2;t_{\epsilon},    t_{ \tilde \epsilon}     )}.
     	\end{split}   
     \end{equation*}
  Indeed if $D(L_{0};m_1,m_2;t_{\epsilon},t_{\tilde \epsilon}) $ occurs then $ Y_{L_{0},m_1,m_2}(t_{\tilde \epsilon})=0 $ and one has an equality. Otherwise, if $D(L_{0};m_1,m_2;t_{\epsilon},t_{\tilde \epsilon}) $ does not occur then 	$ Y_{L_{0},m_1,m_2}(t_{\epsilon})-Y_{L_{0},m_1,m_2}(t_{ \tilde \epsilon} )\geq 0 $ since $Y_{L_{0},m_1,m_2}(t)$ is a decreasing function in $ t $. Now, by \eqref{m: ergodic D int G} and 
  \eqref{dopo}, the last term in \eqref{e: rho difference} is lower bounded by 
       \begin{multline}
       	\label{e: rho nuovo}
       	 \bigl(\rho_{j}(t_\epsilon)-\epsilon_{1}\bigr)\lim_{M \to \infty}\frac{1}{(2M+1)^{2}}\sum_{\substack{m_1,m_2\in \Z:\\|m_1|,|m_2|\leq M}}\mathbf{1}_{G(L_0;m_1,m_2;t_\epsilon,\epsilon_{1})\cap D(L_{0};m_1,m_2;t_{\epsilon},t_{ \tilde \epsilon})}\geq\\
       	\geq \bigl(\rho_{j}(t_\epsilon)-\epsilon_{1}\bigr)\bigl((1-\tilde{\epsilon})p_{cross}-\epsilon_{2}\bigr) \quad \text{a.s.}
       \end{multline}
        Combining \eqref{e: t epsilon} with \eqref{e: rho difference} and \eqref{e: rho nuovo} and  recalling the value of the constants $ \epsilon $, $ \epsilon_{1} $, $ \epsilon_{2} $ and $ \tilde{\epsilon} $, we obtain
     \begin{multline*}
     		\frac{1}{3}p_{cross}\rho_{j}(\infty) =\epsilon>\rho_{j}(t_{\epsilon})-\rho_{j}(\infty)\geq\rho_{j}(t_{\epsilon})-\rho_{j}(  t_{ \tilde \epsilon}   )\geq\\
     		\geq 	\bigl(\rho_{j}(t_\epsilon)-\epsilon_{1}\bigr)\bigl((1-\tilde{\epsilon})p_{cross}-\epsilon_{2}\bigr)\geq 	\bigl(\rho_{j}(\infty)-\epsilon_{1}\bigr)\bigl((1-\tilde{\epsilon})p_{cross}-\epsilon_{2}\bigr)=\\
     		=\frac{1}{2}     p_{cross}\rho_{{j}}(\infty),
     	\end{multline*}
     	which is obviously false when $\rho_{j}(\infty) >0$. 	
     Thus, for each $ i\in \{1,\dots, N\} $ we have that $ \rho_{i}(\infty)=0 $, i.e., 
 each site fixates at the value $ +1 $ (or $ -1 $) from time $ 0 $ with zero probability.
     	This implies, by Lemma \ref{l:it fixates from time 0}, that each site fixates at the value $ +1 $ (or $ -1 $) with zero probability.
     	Hence, all sites flip infinitely often almost surely, i.e., 
     	 the model is of type $ \I $.  
     \end{proof}
     
     \section{Construction of a class of graphs having the planar shrink property and conclusions}
     \label{s: examples}
      We begin by introducing a class of graphs that have  the planar shrink property, as we show in Theorem \ref{p: H planar shrink property}.    
     Let $ \mathcal{H} $ be the collection of infinite plane graphs $ G=(V,E) $ satisfying the following properties:
     \begin{description}
     	\item[(P1)] every edge is a closed  line segment, i.e. a line segment which includes its two endpoints;
 	\item[(P2)] for each $ e\in E $, let us consider the unique straight line $\ell $ which  contains $ e $. Then for any $x \in \ell $ there 
exists $ f\in E $ such that $ x\in f $. 
     	\item[(P3)] $ \Delta(G)<\infty $. 
     \end{description}
     
     
     \begin{figure}[htp]
     	\centering
     	\includegraphics[scale=0.28]{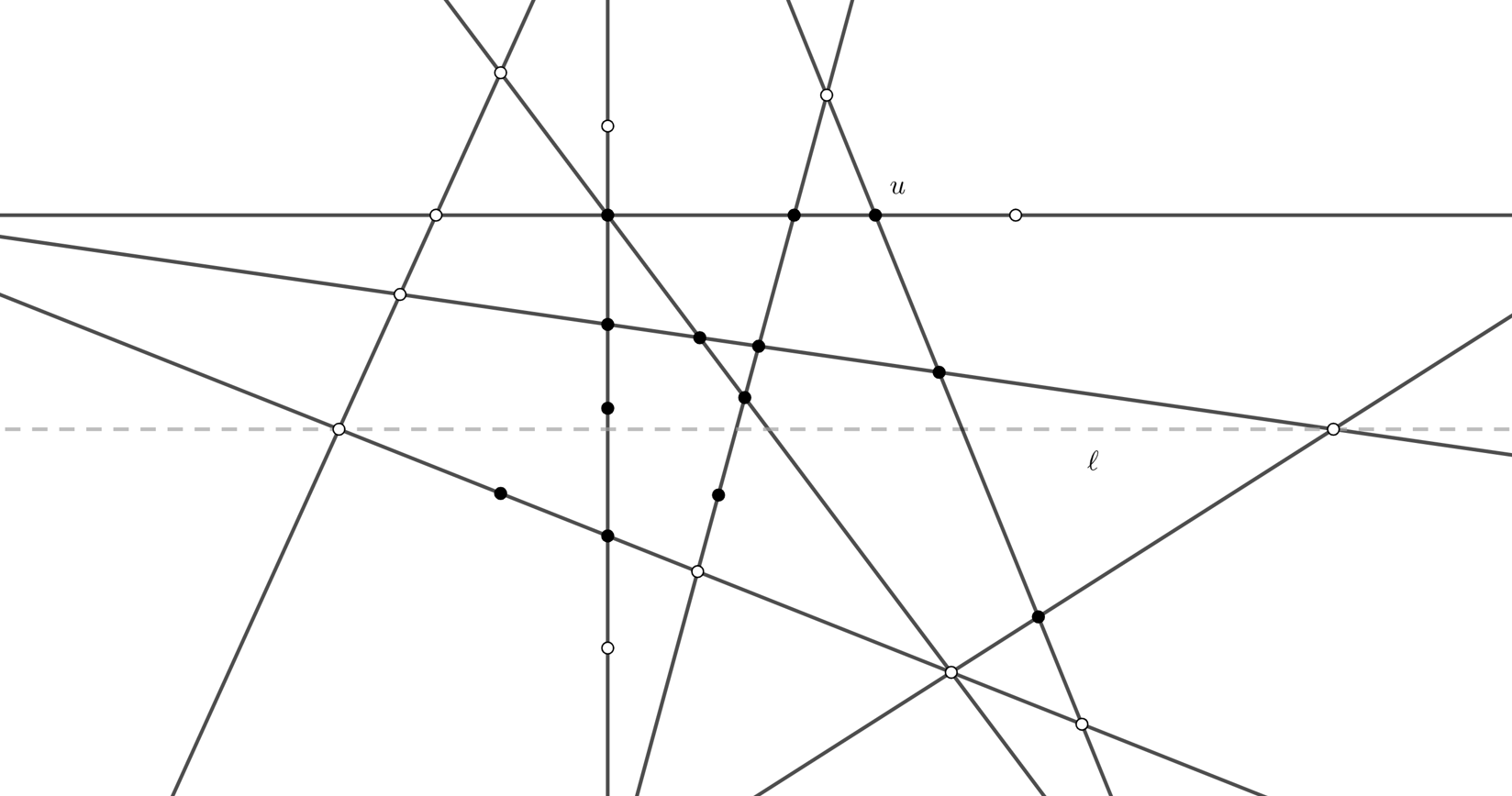}
     	\caption{In black the vertices in $ S $ and in white those in $ V\setminus S $. This figure illustrates a graph  $G \in \mathcal{H}$ and the vertex $u$ used in the proof of  Theorem \ref{p: H planar shrink property}.    		
     	}
     	\label{f:Hshrink}     	     
     \end{figure}

     \begin{theorem}
     	\label{p: H planar shrink property}
     	If $ G\in \mathcal{H} $, then $ G $ has the planar shrink property.
     \end{theorem}
     \begin{proof}
     	Given a non-empty subset $ S\subset V $ and a straight line $ \ell $, suppose without loss of generality that 
 $ 0<|S_1^{\ell}|<\infty $ (see Definition \ref{d:planar shrink property}).   
     	 We need  to prove  that  there exists $ u\in S_1^{\ell} $ such that $ deg_{V\setminus S}(u) \geq deg_{S}(u) $. 
     	First, we note that by property (P2) every vertex has a degree that is even. 
     	Without loss of generality, we can consider the line $ \ell $ coincident with the axis $ x $ (by applying a translation  and  a rotation)  such that  the vertices in $ S_1^{\ell} $ have a non-negative ordinate. 
     	We  write every vertex $ v\in V $ as $ v=(v_x,v_y )\in \R^2 $ and let $ r_{y}:=\max_{v\in S_1^{\ell}}v_y $ and 
     	$ r_{x}:=\max \{v_x \in \R :  (v_x , r_y)\in S_1^{\ell}      \}$.

We consider the vertex  $ u= (r_x, r_y) \in V$.
     	Given $ w =(w_x, w_y)\in N_S(u) $, there are two cases to consider (see Figure \ref{f:Hshrink}). 
     	If $ w_y=r_y $  then, by property (P2), there exists a vertex $ w' =(w'_x, w'_y)\in N_{V\setminus S}(u) $ with $  w'_y=r_y $ and  $ w'_x >r_x $.
     	If instead $ w_y<r_y $  then, by property (P2), there exists a vertex $ w'\in N_{V\setminus S}(u) $ with $ w'_y>r_y $.
     	In both cases,  $ w^\prime $ belongs to the linear extension of edge $ \{u,w\} $ out of $S$.
     	In other words, it is possible to define an injective function
     	\begin{equation*}
     		f_{u}\colon w\in N_S(u)\mapsto w^\prime\in N_{V\setminus S}(u),
     	\end{equation*} 
     	where the vertices $ u $, $ w $ and $ w^\prime $ are aligned.       
     	This implies that $ deg_{V\setminus S}(u)\geq deg_S(u) $.     	      	
     \end{proof}
     
     
    We note that the  shrink property holds even if we  replace $ \mathcal{H} $ with 
 a class of graphs embedded in $ \R^{d} $ having the properties (P1), (P2) and (P3), i.e. they are obtained by intersection of lines. The proof of this fact is analogous to the proof of Theorem \ref{p: H planar shrink property}.      	
     

          \begin{figure}
      	\centering
      	\includegraphics[scale=0.37]{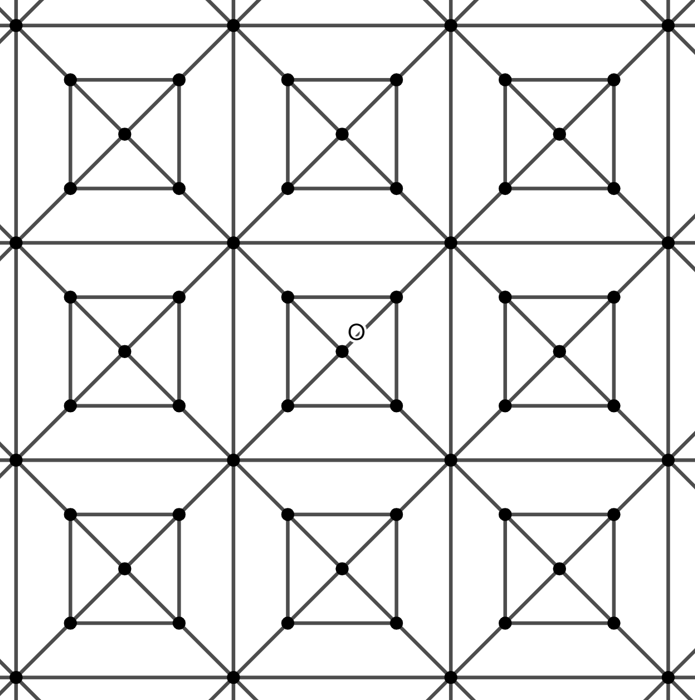} \quad
      	\includegraphics[scale=0.40]{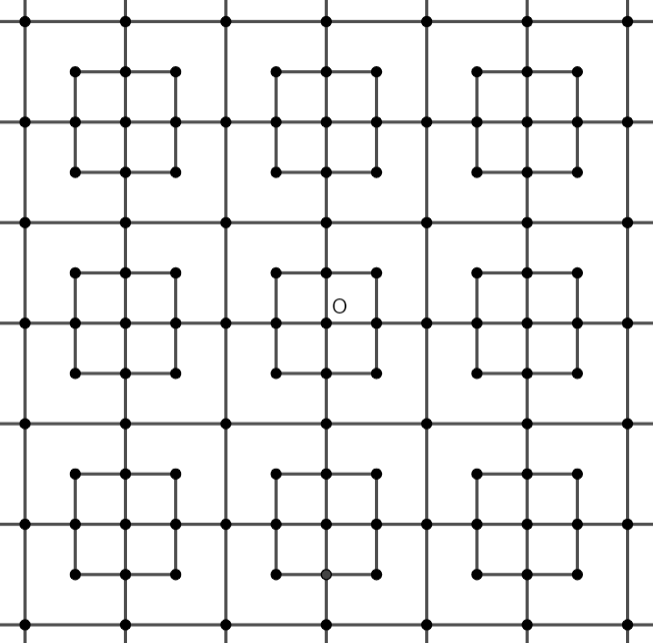}
      	\caption{Examples of graphs in $ \G(4) $ that do not have the shrink property.}
      	\label{f: non shrink}
      \end{figure}
     
 Now,     we  provide some explicit  examples.     
      The  square lattice $ \Z^2 $, the triangular lattice (see Figure \ref{f:trL}) and the graphs in Figures 
      \ref{f: pi/2 e pi}, \ref{f:trL}, \ref{f: double lattice Z^2 con passi 1 e 2} and \ref{f: ex180} belong to class $ \mathcal{H} $  therefore, by Theorem \ref{p: H planar shrink property}, all these graphs have  the planar shrink property. In particular, we note that the graph on the right in Figure \ref{f:trL} belongs to $ \G(3)\setminus \G(6) $, i.e. it is invariant under rotation of an angle of $ 2\pi/3 $, but not of $ \pi/3 $.   In Figure \ref{f: non shrink} we give two examples of graphs that do not have the
 shrink property. In Figure \ref{f: ex180} we show a graph $G $ with infinite classes that is invariant by a rotation of $\pi $ 
but not invariant under a rotation of $\pi/2$. For this  $I(G, 1/2)$-model  we can not apply Theorem \ref{t:main result} but we have a proof showing that it is of type $\I$. We do not present this proof that  would break the unitary character of our presentation.

     \begin{figure}[htp]
     	\centering
     	\includegraphics[scale=0.40]{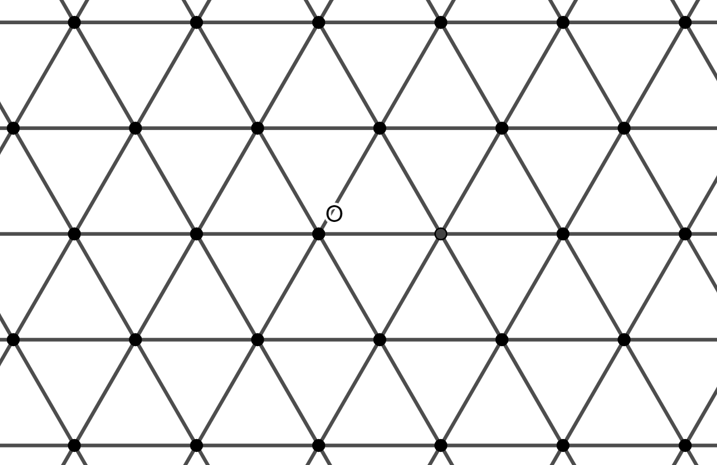}\quad
     	\includegraphics[scale=0.36]{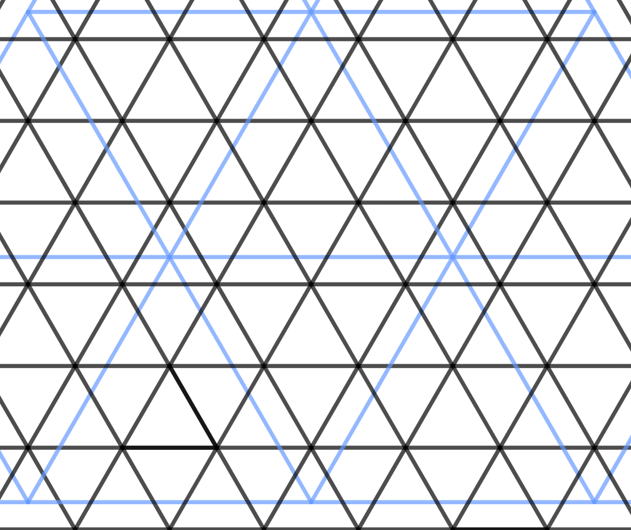}
     	\caption{On the left the triangular lattice, example of a graph $ G\in \G(6)\cap \mathcal{H}$. On the right a double triangular lattice, example of a graph $ G\in (\G(3)\setminus \G(6))\cap \mathcal{H} $. 
     	 }
     	\label{f:trL}
     \end{figure}

     \begin{figure}[h]
     	\centering
     	\includegraphics[scale=0.35]{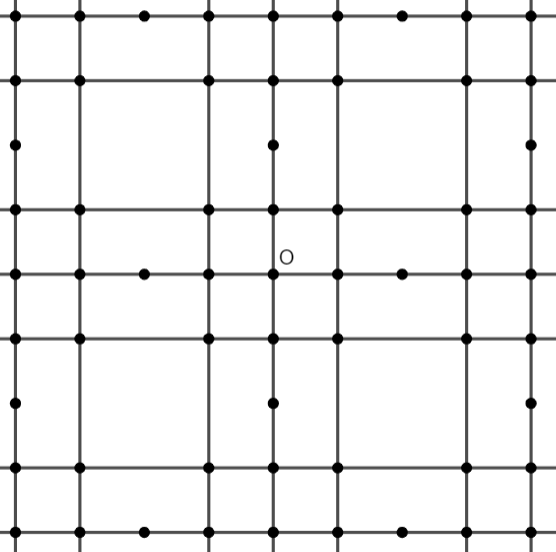}
     	\caption{Modified double lattice $ \Z^2 $: example of a graph $ G\in \G(4)\cap \mathcal{H} $. 
     	}
     	\label{f: double lattice Z^2 con passi 1 e 2}
     \end{figure}
     
     
       \begin{figure}[h]
     	\centering
     	\includegraphics[scale=0.35]{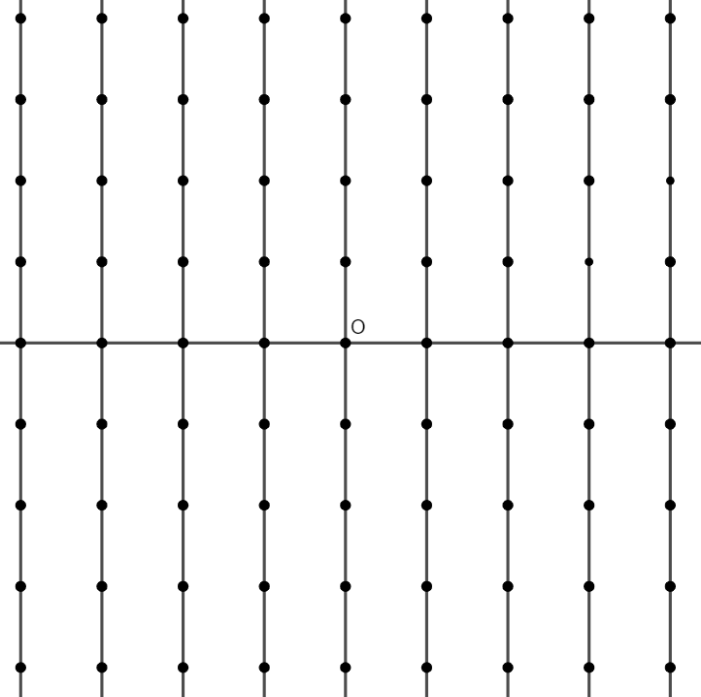}
     	\caption{Example of a graph $ G\in \mathcal{H} $ that is invariant under translation and rotation of $ \pi $, but not of $ \pi/2 $. The number of classes of $ G $ is infinite.  
     	    	}
     	\label{f: ex180}
     \end{figure}

\subsection*{Data Availibility Statement} This mathematical article has no external data to present.

\subsection*{Competing Interests} The authors have no conflicts of interest to declare that are relevant to the content of this article. 
     \subsection*{Acknowledgements} We thank Andrea Maffei  for a useful discussion and comments about Lemma \ref{l:rotations}.


     \bibliographystyle{abbrv}

\begin{thebibliography}{10}
     	
     	
     	\bibitem{A1983}
     	R.~Arratia.
     	\newblock Site recurrence for annihilating random walks on {${\bf Z}_{d}$}.
     	\newblock {\em Ann. Probab.}, 11(3):706--713, 1983.
     	
     	\bibitem{BS2011}
     	I.~Benjamini and O.~Schramm.
     	\newblock Percolation beyond {$\Bbb Z^d$}, many questions and a few answers
     	[mr1423907].
     	\newblock In {\em Selected works of {O}ded {S}chramm. {V}olume 1, 2}, Sel.
     	Works Probab. Stat., pages 679--690. Springer, New York, 2011.
     	
     	\bibitem{BK1989}
     	R.~M. Burton and M.~Keane.
     	\newblock Density and uniqueness in percolation.
     	\newblock {\em Comm. Math. Phys.}, 121(3):501--505, 1989.
     	
     	\bibitem{CDN:clusters}
     	F.~Camia, E.~De~Santis, and C.~M. Newman.
     	\newblock Clusters and recurrence in the two-dimensional zero-temperature
     	stochastic {I}sing model.
     	\newblock {\em Ann. Appl. Probab.}, 12(2):565--580, 2002.
     	
     	\bibitem{CNS2002}
     	F.~Camia, C.~M. Newman, and V.~Sidoravicius.
     	\newblock Approach to fixation for zero-temperature stochastic {I}sing models
     	on the hexagonal lattice.
     	\newblock In {\em In and out of equilibrium ({M}ambucaba, 2000)}, volume~51 of
     	{\em Progr. Probab.}, pages 163--183. Birkh\"{a}user Boston, Boston, MA,
     	2002.
     	
     	\bibitem{CM2006}
     	P.~Caputo and F.~Martinelli.
     	\newblock Phase ordering after a deep quench: the stochastic {I}sing and hard
     	core gas models on a tree.
     	\newblock {\em Probab. Theory Related Fields}, 136(1):37--80, 2006.
     	
     	\bibitem{CDS18}
     	R.~Cerqueti and E.~De~Santis.
     	\newblock Stochastic {I}sing model with flipping sets of spins and fast
     	decreasing temperature.
     	\newblock {\em Ann. Inst. Henri Poincar\'{e} Probab. Stat.}, 54(2):757--789,
     	2018.
     	
     	\bibitem{CCK17}
     	D.~Chelkak, D.~Cimasoni, and A.~Kassel.
     	\newblock Revisiting the combinatorics of the 2{D} {I}sing model.
     	\newblock {\em Ann. Inst. Henri Poincar\'{e} D}, 4(3):309--385, 2017.
     	
     	\bibitem{DEKNS2016}
     	M.~Damron, S.~M. Eckner, H.~Kogan, C.~M. Newman, and V.~Sidoravicius.
     	\newblock Coarsening dynamics on {$\Bbb{Z}^d$} with frozen vertices.
     	\newblock {\em J. Stat. Phys.}, 160(1):60--72, 2015.
     	
     	\bibitem{DKNS2013}
     	M.~Damron, H.~Kogan, C.~M. Newman, and V.~Sidoravicius.
     	\newblock Fixation for coarsening dynamics in 2{D} slabs.
     	\newblock {\em Electron. J. Probab.}, 18:No. 105, 20, 2013.
     	
     	\bibitem{DKNS2016}
     	M.~Damron, H.~Kogan, C.~M. Newman, and V.~Sidoravicius.
     	\newblock Coarsening with a frozen vertex.
     	\newblock {\em Electron. Commun. Probab.}, 21:Paper No. 9, 4, 2016.
     	
     	\bibitem{DSM16}
     	E.~De~Santis and A.~Maffei.
     	\newblock Perfect simulation for the infinite random cluster model, {I}sing and
     	{P}otts models at low or high temperature.
     	\newblock {\em Probab. Theory Related Fields}, 164(1-2):109--131, 2016.
     	
     	\bibitem{DN2003}
     	E.~De~Santis and C.~M. Newman.
     	\newblock Convergence in energy-lowering (disordered) stochastic spin systems.
     	\newblock {\em J. Statist. Phys.}, 110(1-2):431--442, 2003.
     	
     	\bibitem{RD-graph}
     	R.~Diestel.
     	\newblock {\em Graph theory}, volume 173 of {\em Graduate Texts in
     		Mathematics}.
     	\newblock Springer, Berlin, fifth edition, 2017.
     	
     	\bibitem{EN2015}
     	S.~M. Eckner and C.~M. Newman.
     	\newblock Fixation to consensus on tree-related graphs.
     	\newblock {\em ALEA Lat. Am. J. Probab. Math. Stat.}, 12(1):357--374, 2015.
     	
     	\bibitem{FSS2002}
     	L.~R. Fontes, R.~H. Schonmann, and V.~Sidoravicius.
     	\newblock Stretched exponential fixation in stochastic {I}sing models at zero
     	temperature.
     	\newblock {\em Comm. Math. Phys.}, 228(3):495--518, 2002.
     	
     	\bibitem{GNS2000}
     	A.~Gandolfi, C.~M. Newman, and D.~L. Stein.
     	\newblock Zero-temperature dynamics of {$\pm J$} spin glasses and related
     	models.
     	\newblock {\em Comm. Math. Phys.}, 214(2):373--387, 2000.
     	
     	\bibitem{GNS2018}
     	R.~Gheissari, C.~M. Newman, and D.~L. Stein.
     	\newblock Zero-temperature dynamics in the dilute {C}urie-{W}eiss model.
     	\newblock {\em J. Stat. Phys.}, 172(4):1009--1028, 2018.
     	
     	\bibitem{HJ2006}
     	O.~H\"{a}ggstr\"{o}m and J.~Jonasson.
     	\newblock Uniqueness and non-uniqueness in percolation theory.
     	\newblock {\em Probab. Surv.}, 3:289--344, 2006.
     	
     	\bibitem{H1974}
     	T.~E. Harris.
     	\newblock Contact interactions on a lattice.
     	\newblock {\em Ann. Probability}, 2:969--988,1974. 
     	
     	\bibitem{H1977}
     	T.~E. Harris.
     	\newblock A correlation inequality for {M}arkov processes in partially ordered
     	state spaces.
     	\newblock {\em Ann. Probability}, 5(3):451--454, 1977.
     	
     	\bibitem{H1978}
     	T.~E. Harris.
     	\newblock Additive set-valued {M}arkov processes and graphical methods.
     	\newblock {\em Ann. Probability}, 6(3):355--378, 1978.

     	 
     	 \bibitem{L2017}
     	 N.~Lanchier.
     	 \newblock {\em Stochastic modeling}.
     	 \newblock  Universitext. Springer, Cham, 2017.
     	     	
     	\bibitem{L1985}
     	T.~M. Liggett.
     	\newblock {\em Interacting particle systems}, volume 276 of {\em Grundlehren
     		der mathematischen Wissenschaften [Fundamental Principles of Mathematical
     		Sciences]}.
     	\newblock Springer-Verlag, New York, 1985.
     	
     	\bibitem{M1999}
     	F.~Martinelli.
     	\newblock Lectures on {G}lauber dynamics for discrete spin models.
     	\newblock In {\em Lectures on probability theory and statistics
     		({S}aint-{F}lour, 1997)}, volume 1717 of {\em Lecture Notes in Math.}, pages
     	93--191. Springer, Berlin, 1999.
     	
     	\bibitem{M2011}
     	R.~Morris.
     	\newblock Zero-temperature {G}lauber dynamics on {$\Bbb Z^d$}.
     	\newblock {\em Probab. Theory Related Fields}, 149(3-4):417--434, 2011.
     	
     	\bibitem{NNS2000}
     	S.~Nanda, C.~M. Newman, and D.~L. Stein.
     	\newblock Dynamics of {I}sing spin systems at zero temperature.
     	\newblock In {\em On {D}obrushin's way. {F}rom probability theory to
     		statistical physics}, volume 198 of {\em Amer. Math. Soc. Transl. Ser. 2},
     	pages 183--194. Amer. Math. Soc., Providence, RI, 2000.
     	
     	\bibitem{W1991}
     	D.~Williams.
     	\newblock {\em Probability with martingales}.
     	\newblock Cambridge Mathematical Textbooks. Cambridge University Press,
     	Cambridge, 1991.
     	
     \end{thebibliography}

\end{document}